\documentclass [10pt,reqno]{amsart}

\newlength{\myhmargin}  
\usepackage{amsmath,amssymb,amsthm,amsfonts,amscd,flafter,graphicx,verbatim,pinlabel,mathrsfs,enumitem,floatrow,tikz,xcolor,color,tikzscale}
\usetikzlibrary{decorations.pathmorphing}
\usetikzlibrary{decorations.pathreplacing}
\usepackage[all]{xy}
\xyoption{import,dvips,rotate}
\usepackage{epstopdf}
\usepackage[colorlinks=false]{hyperref}
\usepackage[all]{hypcap}

\usepackage[latin1]{inputenc}

\newtheorem{theorem}{Theorem}[section]
\newtheorem{lemma}[theorem]{Lemma}

\newtheorem{proposition}[theorem]{Proposition}
\newtheorem{corollary}[theorem]{Corollary}
\newtheorem{conjecture}[theorem]{Conjecture}

\theoremstyle{definition}
\newtheorem{definition}[theorem]{Definition}
\newtheorem{remark}[theorem]{Remark}
\newtheorem{example}[theorem]{Example}

\renewcommand{\epsilon}{\varepsilon}

\newcommand{\ob}{\text{Ob}}

\newcommand{\gr}[1]{|{#1}|}

\DeclareMathAlphabet{\mathpzc}{OT1}{pzc}{m}{it}
\usepackage{mathrsfs}

\newcommand{\dd}{\textbf{d}}

\newcommand{\Z}{\mathbb{Z}}

\newcommand{\R}{\mathbb{R}}

\newcommand{\E}{\mathbb{E}}

\renewcommand{\phi}{\varphi}

\newcommand{\Hom}{\text{Hom}}

\newcommand{\id}{\text{id}}

\newcommand{\sm}{\setminus}

\renewcommand{\bar}{\overline}

\newcommand{\cC}{\mathcal{C}}
\newcommand{\cM}{\mathcal{M}}

\newcommand{\conf}{\mathscr{C}}

\newcommand{\PP}{{}^+_+}

\newcommand{\MP}{{}^+_-}

\DeclareMathOperator{\Sq}{Sq}

\begin{document}
\parindent0em
\thispagestyle{empty}
\title[Framed cobordism and flow category moves]{Framed cobordism and flow category moves}
\author[Andrew Lobb]{Andrew Lobb}
\address{Department of Mathematical Sciences\\ Durham University\\ United Kingdom}
\email{andrew.lobb@durham.ac.uk}

\author[Patrick Orson]{Patrick Orson}
\address{Department of Mathematical Sciences\\ Durham University\\ United Kingdom}
\email{patrick.orson@durham.ac.uk}

\author[Dirk Sch\"utz]{Dirk Sch\"utz}
\address{Department of Mathematical Sciences\\ Durham University\\ United Kingdom}
\email{dirk.schuetz@durham.ac.uk}

\thanks{The authors were partially supported by the EPSRC grant EP/K00591X/1.}

\begin {abstract}
Framed flow categories were introduced by Cohen--Jones--Segal as a way of encoding the flow data associated to a Floer functional.  A framed flow category gives rise to a CW-complex with one cell for each object of the category.  The idea is that the Floer invariant should take the form of the \emph{stable homotopy type} of the resulting complex, recovering the Floer cohomology as its singular cohomology.  Such a framed flow category was produced, for example, by Lipshitz--Sarkar from the input of a knot diagram, resulting in a stable homotopy type generalizing Khovanov cohomology.

In this paper we give moves that change a framed flow category without changing the associated stable homotopy type.  These are inspired by moves that can be performed in the Morse--Smale case without altering the underlying smooth manifold.  We posit that if two framed flow categories represent the same stable homotopy type then a finite sequence of these moves is sufficient to connect the two categories.  This is directed towards the goal of reducing the study of framed flow categories to a combinatorial calculus.

We provide examples of calculations performed with these moves (related to the Khovanov framed flow category), and prove some general results about the simplification of framed flow categories via these moves.
\end {abstract}

\maketitle

\section{Introduction}
Roughly speaking, a \emph{framed flow category} is a way of encoding the flow data associated to a Morse--Smale vector field or, more generally, the flow data of a Floer functional.  Objects of the category correspond to critical points while the morphism set between any two objects has the structure of a framed manifold-with-corners corresponding to the space of flowlines between the objects.

A framed flow category gives rise to a CW-complex with one cell for each object of the category, and the attaching maps determined by the morphism spaces.  When the framed flow category arises from a Floer functional, the (shifted) cohomology of the CW-complex agrees with the corresponding Floer cohomology.  The idea is that the whole CW-complex (up to suspension and desuspension) should be considered as the Floer invariant, instead of just its cohomology.  In other words, the Floer invariant should take the form of a \emph{stable homotopy type}.


In this paper we construct moves that change a framed flow category without changing the associated stable homotopy type.  These are inspired by moves that can be performed in the Morse--Smale situation without altering the underlying smooth manifold.  We posit that if two framed flow categories represent the same stable homotopy type then a finite sequence of these moves is sufficient to connect the two categories.  This is directed towards the goal of reducing the study of framed flow categories (and their associated stable homotopy types) to combinatorics.

We show in Theorem \ref{thm:SNF} that after a finite sequence of our moves, we can arrange for the associated cochain complex to the framed flow category to be in primary Smith normal form.  This immediately gives Corollary \ref{cor:moore}.  This corollary implies, for example, that if two framed flow categories each represent the same Moore space, then the categories are indeed connected by a finite sequence of moves.

In a later paper we shall go further, developing the combinatorics of the moves and proving results based now on the classifications due to Chang and Baues--Hennes of stable homotopy types of small cohomological width.

\subsection{Plan of the paper}
\label{subsec:plan}
In Section \ref{sec:ffcs}, we give a condensed recap of the necessary definitions of a framed flow category.  This is based on the formulation by Lipshitz and Sarkar \cite{MR3230817}, and is intended as a quick-reference guide for the rest of the paper.  In Section \ref{sec:handleslide} we introduce an analogue for framed flow categories of the operation of \emph{handle sliding} in a handlebody decomposition of a smooth manifold.  In Section \ref{sec:whitney} we similarly introduce an analogue of the \emph{Whitney trick}.

Handle sliding, the extended Whitney trick, and handle cancellation (already introduced in \cite{JLSmorse}) are each moves that can performed on framed flow categories while preserving the associated stable homotopy type.  We give an example of their application in Section \ref{sec:tref}.  Recently, Lipshitz and Sarkar \cite{MR3230817} have constructed a framed flow category from the input of a knot diagram.  This gives a stable homotopy type whose cohomology recovers Khovanov cohomology (a knot invariant intimately connected with Floer cohomology).  Our example is drawn from a particular quantum degree ($q = 21$) of the Khovanov stable homotopy type of the disjoint union of three right-handed trefoils.

Finally, in Section \ref{sec:smith} we prove general results about simplification of framed flow categories.  Essentially, these results can be taken as saying that each framed flow category is move-equivalent to a framed flow category realizing the simplest possible cochain complex giving the correct cohomology.

\section{Framed flow categories}
\label{sec:ffcs}

The definition of a flow category was originally given by Cohen, Jones and Segal \cite{MR1338869} but the version we are working with is based on the definition by Lipshitz and Sarkar \cite{MR3230817}; the equivalence of the two approaches is shown in \cite[Prop.\ 3.27]{MR3230817}.

Before we can give the definition of a flow category, we will need a sharpening of the concept of a smooth manifold with corners. This refinement was originally defined by J{\"a}nich \cite{MR0226674} and was further devloped by Laures \cite{MR1781277}. We refer the reader to these sources for a fuller treatment.

\begin{definition}Let $M$ be a smooth $m$--dimensional manifold with corners. For each $p\in M$ there is an integer $l(p)$ such that an open neighbourhood of $p$ is homeomorphic to $(\R_+)^l\times \R^{m-l}$, with $p$ corresponding to $0\in (\R_+)^l\times \R^{m-l}$. We say $p$ is a \emph{corner point of codimension $l$}. The closure of any connected component of the set $\{p\in M\,|\,l(p)=1\}$ is called a \emph{face} of $M$, and the union of the faces is the \emph{boundary} $\partial M $ of $M$.

Now suppose each $p\in M$ belongs to precisely $l(p)$ faces. In this case, a \emph{smooth $m$--dimensional $\langle n\rangle$--manifold} is a manifold with corners $M$, together with an ordered $n$--tuple of faces $(\partial_1M,\dots,\partial_nM)$ which satisfy the following conditions:

\begin{itemize}
\item $\partial_1M\cup\dots\cup\partial_nM=\partial M$,
\item $\partial_iM\cap\partial_jM$ is a face of $\partial_iM$ and $\partial_jM$ for all $i\neq j$.
\end{itemize}In which case, for any non-empty subset $J\subset \{1,\dots,n\}$, we will write\[\partial_JM=\bigcap_{j\in J}\partial_jM.\]

\end{definition}

Now we introduce some more notation. Let $n$ be a non-negative integer and let $\dd=(d_0,\dots,d_n)$ be an $(n+1)$--tuple of non-negative integers. Write $\R_+=[0,\infty)$ and \[\begin{array}{rclll}\E^\dd&=&\R^{d_0}\times\R_+\times \R^{d_1}\times\R_+\times\dots\times\R_+\times\R^{d_n},&&\\ \E^\dd[a:b]&=&\E^{(d_a,\dots,d_{b-1})},&&\text{for $0\leq a<b\leq n+1$.}\end{array}\]We introduce a standard subdivision of the boundary of $\E^\dd$ into the following faces:
\[\partial_i\E^\dd=\R^{d_0}\times\R_+\times\dots\R^{d_{i-1}}\times\{0\}\times\R^{d_i}\times\dots\times \R_+\times\R^{d_n},\]for $i=1,\dots,n$, so that $\E^\dd$ is a $(d_0+\dots+d_n+n)$--dimensional $\langle n\rangle$--manifold.

\begin{definition}\label{def:neat}A \emph{neat immersion (embedding)} of a smooth $m$--dimensional $\langle n\rangle$--manifold $M$ is a smooth immersion (embedding) $\iota:M \looparrowright \E^\dd$ such that for all $l$, corner points of codimension $l$ in $M$ are sent to corner points of codimension $l$ in $\E^\dd$, and for any $J\subset I\subset \{1,\dots n\}$ the intersection $\partial_IM\cap\partial_J\E^\dd=\partial_J M$ is perpendicular.
\end{definition}

Any compact $m$--dimensional $\langle n\rangle$--manifold admits a neat embedding with respect to some $\E^\dd$ (see Laures \cite[2.1.7]{MR1781277}). Moreover, a neat immersion can always be improved to a neat embedding by perturbation, possibly after increasing the integers in the tuple $\dd$.

\begin{definition}A \em flow category \em $\conf$ is a category with finitely many objects, and equipped with a function $\gr{\cdot}\colon \ob(\conf) \to \Z$, called the \em grading\em, satisfying the following:
\begin{enumerate}
\item $\Hom(x,x)=\{\id\}$ for all $x\in\ob(\conf)$, and for $x\neq y\in\ob(\conf)$, $\Hom(x,y)$ is a smooth, compact $(|x|-|y|-1)$--dimensional $\langle |x|-|y|-1\rangle$--manifold which we denote by $\cM(x,y)$.
\item For $x,y,z\in \ob(\conf)$ with $\gr{z}-\gr{y}=m$, the composition map restricts to
\[
\circ\colon \mathcal{M}(z,y) \times \mathcal{M}(x,z) \to \mathcal{M}(x,y),
\]
which is an embedding into $\partial_m\mathcal{M}(x,y)$. Furthermore,
\[
 \circ^{-1}(\partial_i \mathcal{M}(x,y)) = \left\{ \begin{array}{lr}
                                             \partial_i \mathcal{M}(z,y)\times \mathcal{M}(x,z) & \mbox{for }i<m \\
                                             \mathcal{M}(z,y)\times \partial_{i-m}\mathcal{M}(x,z) & \mbox{for }i>m
                                            \end{array}
\right.
\]
\item For $x\not= y\in \ob(\conf)$, $\circ$ induces a diffeomorphism
\[
 \partial_i\mathcal{M}(x,y) \cong \coprod_{z,\,\gr{z}=\gr{y}+i} \mathcal{M}(z,y) \times \mathcal{M}(x,z).
\]
\end{enumerate}
The manifold $\mathcal{M}(x,y)$ is called the \em moduli space from $x$ to $y$\em, and we also set $\mathcal{M}(x,x)=\emptyset$.
\end{definition}
\begin{definition}
Let $\conf$ be a flow category and write \[k=\min\{|x|\,:\,x\in\ob(\conf)\},\quad n=\max\{|x|\,:\,x\in\ob(\conf)\}-k.\] Suppose $\dd=(d_k,\dots,d_{n+k})$ is an $(n+1)$--tuple of non-negative integers. A \emph{neat immersion (resp.\ neat embedding)} $\iota$ of a flow category $\conf$ relative to $\dd$ is a collection of neat immersions (resp.\ neat embeddings) $\iota_{x,y}:\cM(x,y)\looparrowright\E^\dd[|y|:|x|]$ such that for all objects $x,y,z$ and points $(p,q)\in \cM(x,z)\times\cM(z,y)$,\[\iota_{x,y}(q\circ p)=(\iota_{z,y}(q),0,\iota_{x,z}(p)).\]

If $\iota$ is a neat immersion of a flow category $\conf$ relative to $\dd$, then a \emph{framing} $\phi$ of $(\conf,\iota)$ is a collection of immersions\[\phi_{x,y}:\cM(x,y)\times[-\epsilon,\epsilon]^A\looparrowright \E^\dd[|y|:|x|],\qquad A=d_{|y|}+\dots+d_{|x|-1},\]extending $\iota_{x,y}$, such that for all objects $x,y,z$ and points $(p,q)\in \cM(x,z)\times\cM(z,y)$,\[\phi_{x,y}(q\circ p,t_1,t_2,\dots, t_A)=(\phi_{z,y}(q,t_1,t_2,\dots,t_B),0,\phi_{x,z}(p,t_{B+1},\dots,t_A)),\]where $B= d_{|y|}+\dots+d_{|z|-1}$.
\end{definition}

Note that a (framed) neat immersion of a flow category relative to some $\dd$ can be improved to a (framed) neat embedding relative to some $\dd'\geq\dd$ by perturbation.

\begin{definition} A \emph{framed flow category} is a triple $(\conf,\iota,\phi)$, where $(\conf,\iota)$ is a flow category with neat immersion, relative some $\dd$, and $\phi$ is a framing for $(\conf,\iota)$.
\end{definition}

One can associate to a framed flow category $(\conf, \iota,\phi)$ a geometric realisation as a finite cell complex $|\conf|_{\iota,\phi}=|\conf|$ (see \cite[Definition 3.23]{MR3230817} for precise details of the version we are using). The construction of $|\conf|$ requires several auxilliary choices, one such being an integer $C$ such that $C>|x|$ for all $x\in\ob(\conf)$. To construct $|\conf|$, one takes for each object $x$ of $\conf$ a $(C-m)$--dimensional cell $\mathcal{C}(x)$, where $|x|=m$. These cells are glued together using the data of the framed moduli spaces (essentially via the Thom construction).

The homotopy type of $|\conf|$ depends on the framed embedding $(\iota,\phi)$, but modifying $(\iota,\phi)$ by an isotopy results in a homeomorphic cell complex (see \cite[Lemma 3.25]{MR3230817}). Modifying the other choices involved in constructing $|\conf|$, including possibly increasing $\dd$ to $\dd'\geq\dd$, will result, up to homotopy, in some precise number of reduced suspensions of $|\conf|$ (see \cite{MR3230817}). The following is a well-defined topological object arising from a framed flow category (and independent of choice of $\dd$).

\begin{definition}For a framed flow category $(\conf, \iota, \phi)$ and any object $x$ of $\conf$, let $C$ be the difference between the dimension of the associated cell $\mathcal{C}(x)\subset |\conf|_{\iota,\phi}$ and the grading of $x$. The \emph{stable homotopy type of $(\conf, \iota, \phi)$} is the desuspended suspension spectrum $\mathcal{X}(\conf):=\Sigma^{-C}\Sigma^\infty|\conf|_{\iota,\phi}$, considered up to homotopy of spectra.
\end{definition}

\section{Handle slides in framed flow categories}
\label{sec:handleslide}

In this section, we will define the framed flow category analogue for the process in handlebody theory of sliding one $i$--handle over a second $i$--handle.  This results in a new framed flow category giving rise to the same stable homotopy type.

\begin{theorem}
	\label{thm:handleslide}
	If $(\conf_S,\iota_S,\phi_S)$ is the result of a \emph{handle slide} in $(\conf, \iota, \phi)$, then there is a homotopy equivalence
	\[\mathcal{X}(\conf_S) \simeq \mathcal{X}(\conf) {\rm .}\]
\end{theorem}

We first describe a method for handle sliding, on which we base our flow category analogue.  Given a handlebody $H$, which has $i$--handles $h_1$ and $h_2$, form a new handlebody $H'$ by attaching an $i$--handle $h_3$ and a cancelling $(i+1)$--handle $g$, attached as follows. The attaching sphere for $h_3$ will be glued along the sphere embedded in $H$ formed by taking an embedded connected sum of slightly pushed-off copies of the attaching spheres of $h_1$ and $h_2$. The attaching sphere of $g$ will be glued to intersect the belt spheres of each of our three $i$--handles precisely once positively each (and with no other $i$--handle interactions).

There are then exactly three $i$--handles against which $g$ can be cancelled. The handlebodies which are the effects of cancelling $g$ against $h_1$, $h_2$ and $h_3$ are homeomorphic to, respectively, the effect of $h_1$ slid over $h_2$, the effect of $h_2$ slid over $h_1$, and the original handlebody $H$. The reason for describing handle sliding with such an emphasis on handle cancellation is that it will allow us to build on the flow category handle cancellation technique described in the earlier paper \cite{JLSmorse}.

Let $\conf$ denote a framed flow category $(\conf, \iota, \phi)$ with respect to an $(n+1)$--tuple $\textbf{d}=(d_k,\dots,d_{n+k})$ of non-negative integers. Suppose $x,y$ are objects of $\conf$ with grading $i$. To mimic the geometric process above, we will introduce a pair of cancelling objects $e$ and $f$ in degrees $i$ and $i+1$ respectively. Morphisms will be introduced so that cancelling $f$ against $e$ returns $\conf$, and cancelling $f$ against $x$ returns a new category $\conf_S$ which we will call the \emph{effect of sliding $x$ over $y$}. It is proved in \cite[Theorem 2.17]{JLSmorse} that handle cancellation in flow categories preserves the stable homotopy type of a framed flow category, and hence the stable homotopy types $\mathcal{X}(\conf)$ and $\mathcal{X}(\conf_S)$ will agree.

The main task of this section therefore will be to correctly describe, embed, and frame the \emph{intermediate flow category} obtained by introducing $e$ and $f$.  This is carried out in Subsection \ref{subsec:inter}, with Theorem \ref{thm:handleslide} deduced in Subsection \ref{subsec:sliding}.  We end with a discussion of how handle sliding affects the framings of $1$-dimensional moduli spaces in Subsection \ref{subsec:framings}.

\subsection{The intermediate flow category}
\label{subsec:inter}

We begin with a small lemma that we will use to embed and frame moduli spaces later. Suppose for now that $M^m$ is a smooth $m$--dimensional $\langle n\rangle$--manifold. Then $M\times[0,1]$ has the structure of a smooth $(m+1)$--dimensional $\langle n+1\rangle$--manifold, given by setting $\partial_{n+1}(M\times[0,1])=M\times\{0,1\}$ and $\partial_i(M\times[0,1])=\partial_iM\times[0,1]$ otherwise.

\begin{lemma}\label{lem:arc}Suppose there is a neat immersion $\iota:M\looparrowright \E^{(d_k,\dots,d_{n+k})}$. For real numbers $u>0$, $v\geq 0$, let \[\gamma_{u,v}(t)=(s(t),2t-1)\in \R_+\times\R\] be such that $(t-u)^2+s^2=v^2$. Then the function\[\iota':M\times[0,1]\looparrowright \E^{(d_k,\dots,d_{n+k})}\times \R_+\times\R;\qquad \iota'(p,t):=(\iota(p),\gamma_{u,v}(t))\] is a neat immersion.

Denote by $R:\E^\dd\to \E^\dd$ the reflection in the plane through the origin which is perpendicular to the first coordinate axis of $\R^{d_{n+k}}\subset \E^\dd$. Suppose $\iota$ extends to a neat immersion $\phi:M\times[-\epsilon,\epsilon]^A\looparrowright \E^\dd$, with $A=d_k+\dots+d_{n+k}$, then $-\phi:=R\circ\phi$ determines another possible framing for $\iota$. Define two possible boundary framings for $\iota'|_{M\times\{0,1\}}$by\begin{enumerate}
\item \begin{eqnarray*}\phi_1:M\times\{0,1\}\times[-\epsilon,\epsilon]^{A}\times[-\epsilon,\epsilon]&\looparrowright& \E^\dd\times\R_+\times\R\\
\quad \phi_1(p,0,\textbf{q},r)&=&(\phi(p,\textbf{q}),0, r-1)\\
\phi_1(p,1,\textbf{q},r)&=&(\phi(p,\textbf{q}),0,-r+1)\end{eqnarray*}

\item \begin{eqnarray*}\phi_2:M\times\{0,1\}\times[-\epsilon,\epsilon]^{A}\times[-\epsilon,\epsilon]&\looparrowright& \E^\dd\times\R_+\times\R\\
\quad \phi_2(p,0,\textbf{q},r)&=&(\phi(p,\textbf{q}),0, -r-1)\\
\phi_2(p,1,\textbf{q},r)&=&(-\phi(p,\textbf{q}),0,-r+1)\end{eqnarray*}
\end{enumerate}

Then $\phi_1$ and $\phi_2$ respectively extend to framings for the whole of $\iota'$:\[\phi''_1,\phi_2'':M\times[0,1]\times[-\epsilon,\epsilon]^{A+1}\looparrowright \E^\dd\times\R_+\times\R.\]
\end{lemma}

Actually there is more than one way to extend $\phi_2$ to the required framing $\phi_2''$. The proof below uses a specific choice, which we shall carry with us to make calculations later. Precisely, we make the choice which conforms to one of the cases $\overline{\PP\MP}, \overline{\MP\PP}$ in the \textit{coherent system of paths} described in \cite[Lemma 3.1]{MR3252965}, and in the preceding paragraphs there.

\begin{proof}For a manifold with corners $N$, denote by $N^{(l)}$ the set of corner points of codimension $l$. Then \[(M\times[0,1])^{(l)}=(M^{(l)}\times\{0,1\})\sqcup (M^{(l+1)}\times(0,1)).\]But \[\iota'(M^{(l)}\times\{0,1\})\subset (\E^{(d_k,\dots,d_{k+j})})^{(l)}\times\{0\}\times\R,\]and \[\iota'(M^{(l+1)}\times(0,1))\subset (\E^{(d_k,\dots,d_{k+j})})^{(l+1)}\times(\R_+\setminus\{0\})\times\R,\]so that both components of $M\times[0,1])^{(l)}$ are embedded in $(\E^{(d_k,\dots,d_{k+j})}\times\R_+\times\R)^{(l)}$ as required. Moreover, as $\iota$ was a neat embedding, and the arc $\gamma_{u,v}$ meets the $\R$--axis transversely, $\iota'$ meets the boundary transversality condition (2) of Definition \ref{def:neat} above.

The two different boundary framings $\phi_1$, $\phi_2$ will now be extended to framings of $\iota'$. To build the required framing $\phi''_1$, take the embedding \[\phi':M\times[0,1]\times[-\epsilon,\epsilon]^A\looparrowright \E^\dd\times\R_+\times\R\] and extend it in a direction normal to $\gamma_{u,v}$ in $\R_+\times\R$ so that on the boundary we have $\phi_1$ as required. To build the required framing $\phi_2''$, we must look at the normal direction in $\E^\dd$ that was flipped by $R$ and turn that normal direction through $180^\circ$ as we move along $\gamma_{u,v}$. We do this by framing $\gamma_{u,v}\subset \R\times\R_+\times\R$ as follows. Referring to the coordinates of $\R\times\R_+\times\R$ as $e_1$, $\bar{e}$ and $e_2$, respectively, we will rotate $180^\circ$ around the $e_2$--axis, such that the first vector
equals $\bar{e}$ halfway through (compare \cite[Lemma 3.1]{MR3252965}). The rotation of the 2--dimensional frame as we move along $\gamma_{u,v}$ (extended trivially to the other normal directions) now defines the required embedding \[\phi_2'':M\times[0,1]\times[-\epsilon,\epsilon]^{A+1}\looparrowright \E^\dd\times\R_+\times\R {\rm .}\]
\end{proof}

We now add the objects $e$ and $f$ to $\conf$, along with some new moduli spaces. Schematically, the old and the new morphisms are given by the following diagrams (with grading indicated by horizontal levels):

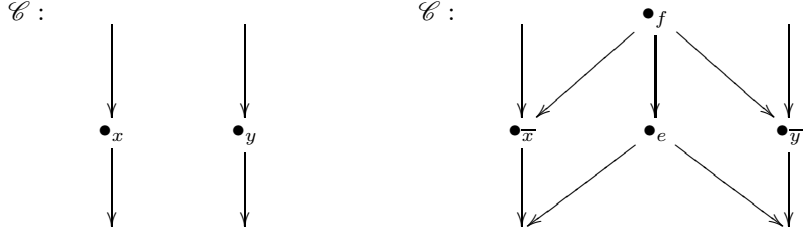
\[\xymatrix @R=1cm @C=0.5cm {\conf:&\ar[d]&&\ar[d]&&&\bar{\conf}:&\ar[d]&&\bullet_f\ar[d]\ar[drr]\ar[dll]&&\ar[d]\\
&\bullet_x\ar[d]&&\bullet_y\ar[d]&&&&\bullet_{\bar{x}}\ar[d]&&\bullet_e\ar[drr]\ar[dll]&&\bullet_{\bar{y}}\ar[d]\\
&&&&&&&&&&&}\]

Specifically, the morphisms and moduli spaces are given as follows.

\begin{definition}\label{def:int}The \emph{intermediate flow category} $\bar{\conf}$ is the flow category with objects
	\[\ob(\bar{\conf})=\{\bar{a}\,|\,a\in\ob(\conf)\}\cup\{e,f\},\]
where $|e|=i$ and $|f|=i+1$ and whose moduli spaces are given by\begin{eqnarray*}
\cM(e,\bar{b})&=&\cM(x,b)\sqcup\cM(y,b),\\
\cM(f,\bar{b})&=&\left\{\begin{array}{lcl}\text{pt.}&&\text{if $\bar{b}=\bar{x},e,\bar{y}$},\\
\cM(e,\bar{b})\times[0,1]&&\text{otherwise}.\end{array}\right.
\end{eqnarray*}In all other cases $\cM(\bar{a},\bar{b})=\cM(a,b)$.
\end{definition}

We now present a neat immersion and two choices of framing for $\bar{\conf}$ relative to $\dd$. With some extra work, it is actually possible to neatly \emph{embed} and frame $\bar{\conf}$ relative to $\dd$. But this is not necessary here as we are free to eventually increase $\dd$ and then perturb the immersion to an embedding, without changing the associated stable homotopy type.

\begin{proposition}\label{prop:immerse}There exists a neat immersion and two framings $(\bar{\iota},\bar{\phi}(+))$, $(\bar{\iota},\bar{\phi}(-))$ of $\bar{\conf}$, relative to $\dd$, which extend the neat embedding and framing $(\iota,\phi)$ of $\conf$.
\end{proposition}

\begin{proof}We will deal with the $(\bar{\iota},\bar{\phi}(+))$ version first and will suppress the ``+'' from the notation of $\bar{\phi}(+)$ while we work on this case. 

When $\bar{a}\neq e,f$, define the embeddings and framings $\bar{\iota}$ and $\bar{\phi}$ of $\cM(\bar{a},\bar{b})$ by the $\iota$ and $\phi$ as in $\conf$.

For $\epsilon$ small, define the embeddings $\bar{\iota}_{f,\bar{x}}$ $\bar{\iota}_{f,e}$, $\bar{\iota}_{f,\bar{y}}$ as the points \[(4\epsilon,0,\dots,0), \quad(0,0,\dots,0), \quad(-4\epsilon,0,\dots,0)\in\R^{d_i}\]respectively. Extend these embeddings to framed embeddings $\bar{\phi}_{f,\bar{x}}$ $\bar{\phi}_{f,e}$, $\bar{\phi}_{f,\bar{y}}$ in the obvious way, by taking the product of $\epsilon$ neighbourhoods in each co-ordinate of $\R^{d_i}$. There are exactly two ways to frame an embedded point in any Euclidean space, which we will denote ``$+$'' and ``$-$''. Declare $\bar{\phi}_{f,\bar{x}}$ and $\bar{\phi}_{f,\bar{y}}$ to both be $+$ framings and $\bar{\phi}_{f,e}$ to be a $-$.

Fixing an object $\bar{b}$ with $|\bar{b}|<i$, we will now immerse and frame $\cM(e,\bar{b})=\cM(x,b)\sqcup\cM(y,b)$. To do this we will simply superimpose the embeddings and framings given by $\iota$ and $\phi$. Note that this might introduce immersion points. Specifically, define the neat immersion\[\bar{\iota}_{e,\bar{b}}:\cM(e,\bar{b})\looparrowright\E^{\textbf{d}}[|e|:|\bar{b}|];\qquad p\mapsto\left\{\begin{array}{lcl}\iota_{x,b}(p)&&\text{if $p\in \cM(x,b)$,}\\ \iota_{y,b}(p)&&\text{if $p\in \cM(y,b)$.}\end{array}\right.\]And similarly, for $\phi_{x,b},\phi_{y,b}$, we obtain the framing $\bar{\phi}_{e,\bar{b}}$.

We will now immerse and frame $\cM(f,\bar{b})$, where $\bar{b}\neq \bar{x}, \bar{y},e$. First apply Lemma \ref{lem:arc}, case (1) to $\iota_{x,b}$, setting $j=i-k-1$ and $(u,v)=(2\epsilon,2\epsilon)$, to obtain neat embeddings $\iota_{x,b}'$, $\phi_{x,b}'$ and $\phi_{x,b}''$. Now apply Lemma \ref{lem:arc}, case (1) a second time, this time to $\iota_{y,b}$, setting $j=i-k-1$ and $(u,v)=(-2\epsilon,2\epsilon)$, to obtain neat embeddings $\iota_{y,b}'$, $\phi_{y,b}'$ and $\phi_{y,b}''$. In order to use these functions to define a neat immersion $\bar{\iota}_{f,\bar{b}}$, we will need to increase the range of $\iota_{y,b}', \iota_{y,b}'$ by a factor or $\R^{d_i-1}$. We do so by simply setting the new $d_i-1$ co-ordinates in the range of each neat embedding to 0. In an abuse of notation, use the same symbols for the new functions - i.e.\ from now on write\begin{eqnarray*}
\iota_{x,b}':&\cM(x,b)\times[0,1]\looparrowright \E^{\textbf{d}}[|x|:|b|]\times\R_+\times\R\times\R^{d_i-1}= \E^{\textbf{d}}[|f|:|\bar{b}|],\\
\iota_{y,b}':&\cM(y,b)\times[0,1]\looparrowright \E^{\textbf{d}}[|y|:|b|]\times\R_+\times\R\times\R^{d_i-1}= \E^{\textbf{d}}[|f|:|\bar{b}|].\end{eqnarray*} Using this, define a neat immersion of $\cM(f,\bar{b})$\[
\bar{\iota}_{f,\bar{b}}:\cM(e,\bar{b})\times[0,1]\looparrowright\E^{\textbf{d}}[|f|,|\bar{b}|];\quad (p,q)\mapsto\left\{\begin{array}{lcl}\iota_{x,b}'(p,q)&&\text{if $p\in \cM(x,b)$,}\\ \iota_{y,b}'(p,q)&&\text{if $p\in \cM(y,b)$.}\end{array}\right.\]To build a framing $\bar{\phi}_{f,\bar{b}}$, we will need to modify \emph{both} the domain and range of the neat embeddings $\phi_{x,b}''$, $\phi_{y,b}''$. Specifically, we must extend the domains of these neat embeddings by a factor of $[-\epsilon,\epsilon]^{d_i-1}$, and we must extend the ranges by a factor of $\R^{d_i-1}$. We do this by simply extending $\phi_{x,b}''$ and $\phi_{y,b}''$ by the obvious inclusion $[-\epsilon,\epsilon]^{d_i-1}\subset\R^{d_i-1}$. Again we abuse notation and use the same symbols for the extended functions:\begin{eqnarray*}
\phi_{x,b}'':&\cM(x,b)\times[0,1]\times[-\epsilon,\epsilon]^A\looparrowright \E^{\textbf{d}}[|x|:|b|]\times\R_+\times\R\times\R^{d_i-1},\\
\phi_{y,b}'':&\cM(y,b)\times[0,1]\times[-\epsilon,\epsilon]^A\looparrowright \E^{\textbf{d}}[|y|:|b|]\times\R_+\times\R\times\R^{d_i-1},\end{eqnarray*}where $A=d_{|\bar{b}|}+\dots+d_{|f|-1}$. So, at last, we have a neat immersion: \[\begin{array}{rcl}\bar{\phi}_{f,\bar{b}}:\cM(e,\bar{b})\times[0,1]\times[-\epsilon,\epsilon]^A&\looparrowright&\E^{\textbf{d}}[|f|,|\bar{b}|]\\(p,q,\textbf{t})&\mapsto&\left\{\begin{array}{lcl}\phi_{x,b}''(p,q.\textbf{t})&&\text{if $p\in \cM(x,b)$,}\\ \phi_{y,b}''(p,q,\textbf{t})&&\text{if $p\in \cM(y,b)$.}\end{array}\right.\end{array}\]Note that the use of the construction from Lemma \ref{lem:arc} is consistent with our choice of $+$ and $-$ framings $\bar{\phi}_{f,\bar{x}}$ $\bar{\phi}_{f,e}$, $\bar{\phi}_{f,\bar{y}}$ from earlier. The ``+'' case is completed.

Now we must describe the required modifications to work on the case $(\bar{\iota},\bar{\phi}(-))$. Again we will suppress the ``$-$'' notation as we work on this case.

The first difference from the $(\bar{\iota},\bar{\phi}(+))$ case is that we declare $\bar{\phi}_{f,\bar{x}}$ be a $+$ framing and both of $\bar{\phi}_{f,\bar{y}}$ and $\bar{\phi}_{f,e}$ to be a $-$ framings. This, in turn, means that for each object $b$, we will want to use case (2) of Lemma \ref{lem:arc} when we come to embed and frame the $\cM(y,b)\times[0,1]$ component of $\cM(f,\bar{b})$. In order to get into case (2) of this lemma, we must endow the $\cM(y,b)$ component of $\cM(e,\bar{b})$ with the framing obtained by taking the image of $\phi_{y,b}$ in $\E^\dd[|y|:|b|]$ and performing the reflection $R$ in the plane through the origin and perpendicular to the first coordinate axis of the $\R^{d_{i-1}}$ factor. Other than these modifications, the construction of the immersions and framings in the $(\bar{\iota},\bar{\phi}(-))$ case proceeds exactly as the $(\bar{\iota},\bar{\phi}(+))$ case.

The $(\bar{\iota},\bar{\phi}(\pm))$ cases are illustrated in the following diagram, where $\alpha$, $\beta$, $\gamma$, $\delta$ denote framings and $-\delta$ represents the effect of reflecting $\delta$ by $R$ as described above.

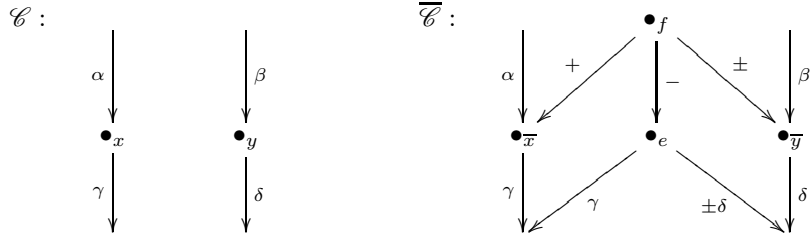
\[\xymatrix @R=1cm @C=0.5cm {\conf:&\ar[d]_-{\alpha}&&\ar[d]^-{\beta}&&&\bar{\conf}:&\ar[d]_-\alpha&&\bullet_f\ar[d]^-{-}\ar[drr]^-{\pm}\ar[dll]_-{+}&&\ar[d]^-{\beta}\\
&\bullet_x\ar[d]_-{\gamma}&&\bullet_y\ar[d]^-{\delta}&&&&\bullet_{\bar{x}}\ar[d]_-{\gamma}&&\bullet_e\ar[drr]_-{\pm\delta}\ar[dll]^-{\gamma}&&\bullet_{\bar{y}}\ar[d]^-{\delta}\\
&&&&&&&&&&&}\]
\end{proof}

\begin{proposition}\label{prop:choices}Cancelling $e$ against $f$ in $(\bar{\conf},\bar{\iota},\bar{\phi}(\pm))$, using the handle cancellation technique of \cite[\textsection 2.3]{JLSmorse}, results in precisely $(\conf, \iota, \phi)$.
\end{proposition}

\begin{proof}Write $(\conf_H,\iota_H,\phi_H)$ for the result of cancelling $e$ against $f$ as described in \cite[\textsection 2.3]{JLSmorse}. According to \cite[Definition 2.9]{JLSmorse}, the objects of $\conf$ and of $\conf_H$ are in 1:1 correspondence. Moreover, recall from \cite[\textsection 2.3]{JLSmorse} that if $a,b$ are objects of $\conf$, then the moduli spaces of $\conf_H$ are given by\[\cM(a,b)=\cM(\bar{a},\bar{b})\cup_f\left(\cM(f,\bar{b})\times \cM(\bar{a},e)\right),\]where $f$ is a certain identification described generally in \cite[\textsection 2.3]{JLSmorse}. However, when $a$ is an object of $\conf$, there are no non-empty moduli spaces $\cM(\bar{a},e)$, so in fact the moduli spaces of $\conf_H$ are the same as those of $\conf$. This simplification, carried through the construction in \cite[Theorem 2.17]{JLSmorse}, actually results in a precise agreement of embeddings and framings so that in fact $(\conf,\iota,\phi)=(\conf_H,\iota_H,\phi_H)$.
\end{proof}

\subsection{Handle sliding in framed flow categories}\label{subsec:sliding}

\begin{definition}The framed flow category $(\conf_S,\iota_S,\phi_S(\pm))$, called the \emph{effect of $(\pm)$--sliding $x$ over $y$ in $(\conf,\iota,\phi)$} is defined as the result of cancelling $f$ against $\bar{x}$ in $(\bar{\conf},\bar{\iota},\bar{\phi}(\pm))$. We will often suppress the ``$+$'' and ``$-$'' from the notation.
\end{definition}

The following lemma is a straightforward consequence of the constructions so far and the definition of handle cancellation given in \cite[Definition 2.9]{JLSmorse}.

\begin{lemma}The framed flow category $\conf_S$ has objects\[\ob(\conf_S)=\{a'\,|\,a\in\ob(\conf)\},\]and moduli spaces given by\begin{eqnarray*}
\cM(x',b')&=&\cM(x,b)\sqcup\cM(y,b),\\
\cM(a',y')&=&\cM(a,x)\sqcup\cM(a,y),
\end{eqnarray*}In all other cases \[\cM(a',b') = \cM(a,b) \sqcup \left(\cM(y,b) \times [0,1] \times \cM(a,x)\right).\]
\end{lemma}

\begin{proof}The main thing to note is that after we cancel $f$ against $\bar{x}$ in the intermediate flow category, we are left with object set $\{\bar{a}\,|\,a\in\ob(\conf)\sm\{x\}\}\cup\{e\}$. Hence the object $x'\in\ob(\conf_S)$ in the statement of the lemma is a relabelling of $e$. The rest of the lemma follows directly from \cite[Definition 2.9]{JLSmorse}.
\end{proof}

The following is a schematic for sliding $x$ over $y$, where the framings determined by a $(\pm)$--slide are also shown. The signs are calculated by following the construction through the handle cancellation in \cite[Definition 2.9]{JLSmorse}.

\[\xymatrix @R=1cm @C=0.5cm {\conf:&\ar[d]_-{\alpha}&&\ar[d]^-{\beta}&&&\conf_S:&\ar[d]_-{\alpha}\ar[drr]^-{\mp\alpha}&&\ar[d]^-{\beta}\\
&\bullet_{x}\ar[d]_-{\gamma}&&\bullet_{y}\ar[d]^-{\delta}&&&&\bullet_{x'}\ar[d]_-{\gamma}\ar[drr]_-{\pm\delta}&&\bullet_{y'}\ar[d]^-{\delta}\\
&&&&&&&&&}\]

Now we are in a position to deduce the main result of this section.

\begin{proof}[Proof of Theorem \ref{thm:handleslide}] Define the intermediate framed flow category $(\bar{\conf},\bar{\iota},\bar{\phi}(\pm))$ as in Definition \ref{def:int} and Proposition \ref{prop:immerse}. By Proposition \ref{prop:choices}, the framed flow category obtained from cancelling $e$ against $f$ is precisely $(\conf, \iota,\phi)$. By sufficiently increasing $\dd$ to some $\bar{\dd}$ we may perturb the framed neat immersions $(\bar{\iota},\bar{\phi}(\pm))$, $(\iota_S,\phi_S)$ to framed neat embeddings and hence geometrically realise the CW-complexes $|\bar{\conf}|$ and $|\conf_S|$. It is immediately clear from \cite[Theorem 2.17]{JLSmorse} that there are homotopy equivalences \[\mathcal{X}(\conf)\simeq\mathcal{X}(\bar{\conf})\simeq\mathcal{X}(\conf_S).\]
\end{proof}

\subsection{Framings of 1--dimensional moduli spaces}\label{subsec:framings}

In order to compute cohomology operations such as $\Sq^2$ we need to understand how the framings of the 1--dimensional moduli spaces behave after a handle slide. If we keep in mind the cancellation from the intermediate flow category $\bar{\conf}$ to $\conf_S$ we get
\begin{align*}
\cM(a',x') &= \cM(f,e) \times \cM(a,x), \\
\cM(a',y') &= \cM(a,y) \sqcup \left(\cM(f,y) \times \cM(a,x)\right).
\end{align*}
Furthermore,
\[
\cM(x',b') = \cM(x,b) \sqcup \cM(y,b),
\]
but with the framing of $\cM(y,b)$ depending on whether we do a $(+)$ or a $(-)$-handle slide.

Since $\cM(f,e)$ is a negatively framed point, we get that the framing on $\cM(a',x')$ corresponds to the old framing on $\cM(a,x)$ by the framings in the cancelled category, compare \cite{JLSmorse}. Similarly, if we use the $(-)$--handle slide the framings on $\cM(a',y')$ agree with the original framings on $\cM(a,y)$ and $\cM(a,x)$. In the case of the $(+)$--handle slide however, the copy of $\cM(a,x)$ in $\cM(a',y')$ has an extra reflection in one coordinate.

Recall that interval components $J$ in $\cM(a,b)$ with $|a|=|b|+2$ can be framed in two ways (up to fixed boundary framing). The standard framings are described in \cite[\S 3.2]{JLSmorse} and \cite[\S 3]{MR3252965}, and we write $fr(J)\in \Z/2\Z$ with $fr(J)=0$ if the framing corresponds to the standard choice, and $fr(J)=1$ if not.

Similarly, if $|a|=|b|+1$ and $A\in \cM(a,b)$ is a point, we write $\varepsilon_A\in \Z/2\Z$ for the framing sign of this point.

\begin{proposition}\label{prop:framechanges}
 Let $\conf$ be a framed flow category containing two objects $x,y$ with the same grading, and let $\conf_S$ be the framed flow category obtained from $\conf$ by sliding $x$ over $y$. Let $a,b$ be objects in $\conf$ with $|a|=|b|+2$.
\begin{enumerate}
 \item If $|a|=|x|+2$, then $\cM(a',b')=\cM(a,b)$ for all $b\not=y$ with the same framing. Furthermore
\[
 \cM(a',y') = \cM(a,y) \sqcup \cM(a,x)
\]
with components from $\cM(a,y)$ identically framed. If $J\subset \cM(a,x)$ is an interval component, denote $J'$ for the same component when viewed as a subset of $\cM(a',y')$. Then
\begin{enumerate}
\item If $\conf_S$ is obtained by a $(-)$--handle slide, then
\[
 fr(J')= fr(J).
\]
\item If $\conf_S$ is obtained by a $(+)$--handle slide, then
\[
 fr(J')= 1 + fr(J).
\]
\end{enumerate}
\item If $|a|=|x|+1$, then
\[
 \cM(a',b') = \cM(a,b) \sqcup \cM(y,b)\times [0,1] \times \cM(a,x)
\]
and every component of $\cM(a,b)$ is framed the same way. For each pair of points $(B,A)\in \cM(y,b)\times \cM(a,x)$, denote $I_{B,A}$ the corresponding interval component in $\cM(a',b')$.
\begin{enumerate}
\item If $\conf_S$ is obtained by a $(-)$--handle slide, then
\[
 fr(I_{B,A})=\varepsilon_B.
\]
\item If $\conf_S$ is obtained by a $(+)$--handle slide, then
\[
 fr(I_{B,A})=1+\varepsilon_B.
\]
\end{enumerate}
\item If $|a|=|x|$ and $a\not= x$, then the framings of components in $\cM(a',b')$ are the same as in $\cM(a,b)$. For 
\[
 \cM(x',b')=\cM(x,b) \sqcup \cM(y,b)
\]
the framing values in the new moduli space agree with the framing values in the old moduli spaces.
\end{enumerate}
\end{proposition}

\begin{proof}
Assume $|a|=|x|+2$. In $\conf_S$ the object $x'$ corresponds to $e$ from the intermediate flow category, and the intervals in $\cM(\bar{a},e)$ are of the form $\{-\} \times J$ where $-\in \cM(f,e)$ from the intermediate flow category, and $J$ is an interval in the original $\cM(a,x)$. By \cite[Proposition 3.6.6]{JLSmorse}, $fr(\{-\}\times J)=fr(J)+1+\varepsilon_-$. This means $fr(\{-\}\times J)=fr(J)$.

For $\cM(a',y')$ the intervals in $\cM(a,y)$ do not change, but we also get intervals of the form $\{B\} \times J$ with $B\in \cM(f,\bar{y})$ and $J$ in $\cM(a,x)$. As before, $fr(\{B\}\times J) = fr(J) + 1 + \varepsilon_B$. Now if we use a $(-)$--slide, $\varepsilon_B=1$, and if we use a $(+)$--handle slide, $\varepsilon_B=0$. Hence we get (1).

Now assume $|a|=|x|+1$. In the intermediate flow category we get intervals $J_C$ and $J_B$ in $\cM(f,\bar{b})$ coming from points $C\in \cM(x,b)$ and $B\in \cM(y,b)$. When passing to $\conf_S$ the intervals $\{A\}\times J_C$ give a collar neighborhood to an interval in $\cM(a,b)$ with endpoint $(C,A)\in \cM(x,b)\times \cM(a,x)$, which do not change the framing of the original interval. The intervals $J_B$ are framed depending on whether we have a $(+)$ or a $(-)$--slide. In the case of a $(+)$--slide we use Lemma \ref{lem:arc}.1. 
Note that the sign of the second coordinate changes, while the first coordinate remains $\varepsilon_B$. But during the rotation described in the proof of Lemma \ref{lem:arc}, the second vector equals $-\bar{e}$. By the choice of a coherent system of paths as done in \cite{MR3252965} or \cite{JLSmorse}, this is a standard path if and only if $\varepsilon_B=1$. Hence $fr(J_B)=1+\varepsilon_B$. Passing to $\conf_S$ gives intervals $I_{B,A}$ for each $A\in \cM(a,x)$ with the same framing value as $J_B$ by \cite[Proposition 3.7.6]{JLSmorse}.

If we perform a $(-)$--slide, we use Lemma \ref{lem:arc}.2, and the sign of the first coordinate changes, while the second coordinate remains. In the proof of Lemma \ref{lem:arc} we note that there is a rotation around the $e_2$--axis chosen such that the first standard vector points in the positive direction of $\R_+$. Comparing with the choice of coherent system of paths we get a standard path if and only if the first vector of the framing points in the positive direction, that is, when $\varepsilon_B=0$. This implies $fr(I_{B,A})=\varepsilon_B$.

Finally, assume that $|a|=|x|$. If $a\not=x$, then $\cM(a',b')=\cM(a,b)$ and the framings do not change. If $a=x$, we get $\cM(x',b')=\cM(x,b)\sqcup \cM(y,b)$ and the framing of $\cM(x,b)$ does not change. If we perform a $(+)$--slide, the framing of $\cM(y,b)$ does not change either. In the case of a $(-)$--slide, the first coordinate of $\R^{d_{i-1}}$ is reflected. This means that the second coordinate of the framing of an interval is changing sign. However, flipping the sign of the second coordinate maps standard paths to standard paths, compare also \cite[Lemma 3.3]{MR3252965}, although only half of this statement is proven there.
Hence the framing does not change either.
\end{proof}

\section{The extended Whitney trick in framed flow categories}
\label{sec:whitney}

In \cite[Theorem 3.1.5]{MR1781277}, Laures extends the Pontryagin--Thom Theorem to the setting of manifolds with corners using a suitably defined framed cobordism category. As the cell attachment maps in the Cohen--Jones--Segal construction are defined by a version of the Pontryagin--Thom collapse map, the homotopy classification of these attaching maps (hence the eventual homeomorphism type of the CW-complex) is closely related to the framed cobordism classes of the moduli spaces in a flow category.

In this section we will re-encode this relationship by defining a new technique for modifying framed flow categories, called \emph{the extended Whitney trick} (compare the Whitney trick of \cite[\textsection 1]{JLSmorse}). Suppose $(\conf, \iota, \phi)$ is a framed flow category and that for some objects $x\neq y$ of $\conf$, that $\cM(x,y)=M$ is a manifold with boundary. Let $W$ be a framed cobordism (rel.\ boundary) between $M$ and another manifold with boundary $M'$. We will show how to use $W$ to define a new framed flow category $(\conf_W,\iota_W,\phi_W)$ in which $M$ is replaced by $M'$ (and the other moduli spaces in $\conf$ are modified appropriately). Using the ideas of the Pontryagin--Thom theorem we will deduce the following.

\begin{theorem}
	\label{thm:whitney}
		If $(\conf_W,\iota_W,\phi_W)$ is the result of an \emph{extended Whitney trick} in $(\conf, \iota, \phi)$, then there is a homotopy equivalence \[\mathcal{X}(\conf_W) \simeq \mathcal{X}(\conf) {\rm .}\]
\end{theorem}

\subsection{Framed cobordism of manifolds with corners rel.\ boundary}We will not need Laures's full machinery \cite{MR1781277} of cobordism of manifolds with corners in the sequel. The eventual complicated interactions between the moduli spaces in a framed flow category mean that allowing an unrestricted framed cobordism in the sense of \cite{MR1781277} becomes intractable. Instead we will work with the following.

\begin{definition}Suppose $M$ and $M'$ is are $m$--dimensional $\langle n\rangle$--manifolds with $\partial_iM=\partial_iM'$ for $i=1,\dots,n$. An $(m+1)$--dimensional $\langle n+1\rangle$--manifold $W$ is called a \emph{cobordism rel.\ boundary} between $M$ and $M'$ if $\partial_{n+1}W=M\sqcup M'$ and $\partial_iW=\partial_i\times[0,1]$ for $i\neq n+1$.

Suppose an embedding $\tilde{\iota}:W\hookrightarrow \E^{\dd}\times[0,1]$, of a cobordism rel.\ boundary, meets $\E^\dd\times\{0,1\}$ orthogonally in $M\sqcup M'$, and induces neat embeddings \[\begin{array}{rrcl}\tilde{\iota}|_M:&M&\hookrightarrow& \E^\dd\times\{1\},\\
\tilde{\iota}|_{M'}:&M'&\hookrightarrow& \E^\dd\times\{0\}.\end{array}\]Then $\tilde{\iota}$ is an \emph{embedded cobrdism rel.\ boundary} between the neat embeddings $(M',\tilde{\iota}|_{M'})$ and $(M,\tilde{\iota}|_{M})$. Suppose furthermore that there exists a framing $\tilde{\phi}$ of such a $(W,\tilde{\iota})$ and that $\tilde{\phi}$ meets $\E^\dd\times\{0,1\}$ orthogonally. Then the framing $\tilde{\phi}$ determines framings $\phi$ and $\phi'$ of $(M,\tilde{\iota}|_M)$ and $(M',\tilde{\iota}|_{M'})$ respectively. If such a $(W,\tilde{\iota}, \tilde{\phi})$ exists, it is called a \emph{framed cobordism rel.\ boundary} between $(M',\tilde{\iota}|_{M'},\phi')$ and $(M,\tilde{\iota}|_{M},\phi)$. $(M',\tilde{\iota}|_{M'},\phi')$ and $(M,\tilde{\iota}|_{M},\phi)$ are called \emph{framed cobordant rel.\ boundary} if there exists a framed cobordism rel.\ boundary between them, possibly after enlarging the $\dd\leq\dd'$.
\end{definition}

\begin{remark}Examples of framed cobordisms rel.\ boundary $W\subset \E^{\dd'}\times[0,1]$ whose framed boundary $M\sqcup M'$ can be framed embedded in a smaller space $\E^\dd$ are fairly common. This is why we have allowed the possibility of enlarging the ambient space in the final definition above. For instance, the generator $1\in\Omega_1^{fr}\cong\Z/2\Z$ can be embedded, along with a framed normal neighbourhood, as $\phi:S^1\times D^2\hookrightarrow\R^3$. But the minimum embedding dimension for a framed nullcobordism of $\phi\sqcup\phi$ is 5.\end{remark}

\subsection{Pushing $M$ out of the corner}

We now make a slight digression into an easy but technical construction we will need later. Suppose $(W,\tilde{\iota},\tilde{\phi})$ is a framed embedded cobordism rel.\ boundary between $m$--dimensional $\langle m\rangle$--manifolds $(M',\iota',\phi')$ and $(M, \iota,\phi)$, where $\iota, \iota'$ are with respect to some $\dd=(d_0,d_1,\dots,d_m)$. Denote by $U_\eta\subset (\R_+)^N$ the open ball at the origin with (small) radius $\eta>0$, $\overline{U_\eta}$ the corresponding closed ball, and by $H_\eta=\overline{U_\eta}\sm U_\eta$. Suppose that \[\iota_X:X\hookrightarrow \E^{\dd}\times(\R_+)^N\]is a neat embedding of an $(m+N)$--dimensional $\langle m+N\rangle$--manifold, with framing $(X,\iota_X,\phi_X)$. Moreover, suppose that near the `corner' $\E^\dd\times\textbf{0}$ this embedding is\[\iota_X(X)\cap \E^\dd\times U_{2\eta}=\iota(M)\times U_{2\eta}\]and that here the framing $\phi_X$ agrees with the framing \[M\times U_{2\eta}\times[-\epsilon,\epsilon]^A\hookrightarrow \E^\dd\times U_{2\eta};\qquad (p,q,\textbf{t})\mapsto (\phi(p,\textbf{t}),q),\]where $A=d_0+\dots+d_m$.

Later on we will need a mechanism to `push $M$ out of the corner' and replace it with $M'$. Roughly speaking, this is achieved by glueing together the two spaces\begin{eqnarray*}X'&=&\iota_X(X)\sm (M\times U_\eta),\\Y&=&(M'\times \overline{U_\eta})\cup_{M'\times H_\eta}(W\sm(M'\times [0,\eta))\times H_\eta,\end{eqnarray*}and then embedding the result appropriately (compare Figure \ref{fig:push}). But to ensure the smooth structures can be made to agree, we will need to use the standard technique of overlaying open collar neighbourhoods.

\begin{figure}
\begin{tikzpicture}
\draw [<->] (0,5) -- (0,0) -- (-5,0);
\draw (-2,0) arc [radius=2, start angle=180, end angle= 90];
\draw (-4,0) arc [radius=4, start angle=180, end angle= 90];
\node at (-4.9,-0.4) {$\R_+$};
\node at (0.4,4.9) {$\R_+$};
\node at (0.3,4) {$2\eta$};
\node at (0.25,2) {$\eta$};
\node at (-1.5,2.5) {$M\times(U_{2\eta}\sm U_\eta)$};
\node at (-0.9,0.7) {$M\times U_\eta$};

\node at (1.65,3.35) {Push $M$ out } ;
\node at (1.65,3)  {of the corner};
\draw [->] [decorate, decoration={snake}] (0.7,2.5) -- (2.6,2.5);

\draw [<->] (7,5) -- (7,0) -- (2,0);
\draw (5,0) arc [radius=2, start angle=180, end angle= 90];
\draw (3,0) arc [radius=4, start angle=180, end angle= 90];
\draw [dashed]  (3.3,0) arc [radius=3.7, start angle=180, end angle= 90];
\draw [dashed]  (2.7,0) arc [radius=4.3, start angle=180, end angle= 90];
\draw [red] (5.0681,0.5176) -- (3.1362, 1.0352);
\draw [red] (5.2679,1) -- (3.5358, 2);
\draw [red] (5.5857,1.4142) -- (4.1715,2.8284);
\draw [red] (6,1.732) -- (5, 3.4641);
\draw [red] (6.4823,1.9318) -- (5.9647,3.8637);
\node at (2.1,-0.4) {$\R_+$};
\node at (7.4,4.9) {$\R_+$};
\node at (7.3,4) {$2\eta$};
\node at (7.25,2) {$\eta$};
\node at (6.1,0.7) {$M'\times U_\eta$};
\node [red] at (3.7,4.8) {$(W\sm M'\times[0,\eta))\times H_\eta$};
\draw [->] (3.7,4.4) to [out=270,in=140] (5,2.6);
\draw [decorate,decoration={brace,amplitude=4pt,raise=2pt}] (3.3,-0.1) -- (2.7,-0.1);
\node at (4.15,-0.6) {\small $M\times H_\eta\times(-\eta',\eta')$};

\end{tikzpicture}
\caption{Example of pushing $M$ out of the corner. Here $N=2$ and $\E^\dd$ is thought of as the direction perpendicular to the page. Each radial red line is a copy of $W$ (minus a collar at the $M'$ boundary).}\label{fig:push}
\end{figure}
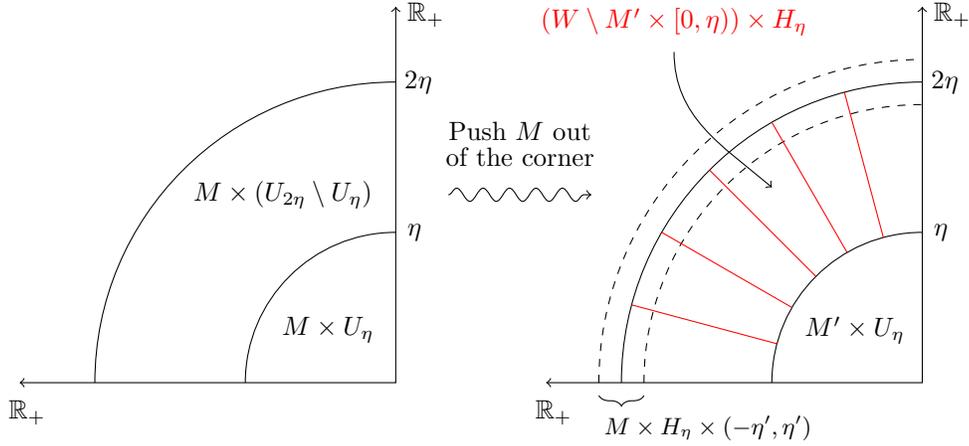

\begin{lemma}\label{lem:smooth}With notation as above, the $(m+N)$--dimensional topological manifold with corners $X'\cup_{M\times H_\eta} Y$ is homeomorphic to a smooth $(m+N)$--dimensional $\langle m+N\rangle$--manifold $Z$, which has a neat embedding $Z\hookrightarrow \E^\dd\times (\R_+)^N$ such that $M'$ is in the corner $\E^\dd\times \overline{U_\eta}$. Moreover the embedding of $Z$ can be framed compatibly with the framings of $(W,\tilde{\iota},\tilde{\phi})$ and of $(X, \iota_X,\phi_X)$.
\end{lemma}

In practice, we will only ever make use of the cases $N=1,2,3$ but it is no extra effort to prove Lemma \ref{lem:smooth} in generality. Also note that if $M$ is in a codimension $N$ corner then it is also in $N$ codimension $N-1$ corners. In fact the restriction of the construction in the proof to some choice of $(\R_+)^{N-1}\subset (\R_+)^N$ precisely reproduces the proof in this lower codimension. We will use this fact later.

\begin{proof}The topological manifold with corners $Y$ already carries a smooth structure as the two components that are glued to make $Y$ may be given their respective product smooth structures, and near the glueing boundary $M'\times H_\eta$ these product smooth structures agree precisely with one another.

Any radial direction of $(\R_+)^N$ determines a copy of $\E^\dd\times \R_+$. We may embed $\iota_Y:Y\hookrightarrow\E^\dd\times(\R_+)^N$ such that in any radial direction we have rescaled the embedding $\tilde{\iota}(W)\subset \E^\dd\times[0,1]\subset \E^\dd\times\R_+$ by a factor of $\eta$. Choose another small $\eta'<<\eta$. There is now a collar neighbourhood $\iota_Y(M\times H_\eta)\times(-\eta',0]$ of the embedded boundary component $M\times H_\eta\subset Y$, where the collar direction is radial in $(\R_+)^N$. Similarly, the boundary component $M\times H_\eta\subset X'$ has a radial collar neighbourhood $M\times H_\eta\times[0,\eta')\subset Y$. Writing $(p, t)\sim(p,(1-t)/2)$ for $p\in M\times H_\eta'$ and $t\in(0,\eta'/2)$, define the identification space\[Z=X'\sqcup Y/\sim.\]The smooth structures on $Y$ and $X'$ can now be made compatible on this open collar overlay, so that $Z$ has the structure of a smooth $\langle m+N\rangle$--manifold. Moreover, by radially dilating the embedding $\iota_X|_{X'}$ of $X'$, and combining with the embedding of $Y$, $Z$ is seen to have a (neat) embedding in $\E^\dd\times\R_+\times\R_+$ so that $M'$ is in the corner as required.

Now we turn to the framings. The framing of the corner embedding of $M'\times \bar{U_\eta}$ is given by extending the framing $\phi'$ of $\iota':M\hookrightarrow \E^\dd$ trivially to the product. Similarly we may frame the embedding of the product $(W\sm(M'\times [0,\eta)))\times H_\eta$ by extending $\tilde{\phi}$. Note that these framings together form a framing of the embedded $Y$. Now, we assumed that $\phi_X$ agreed with $\phi$ near the corner $\E^\dd\times \textbf{0}$. Hence to frame $Z$, we can simply glue the framing on $X'$ coming from $X$ to the framing on $Y$ as they precisely agree on the collar overlap.
\end{proof}

\subsection{The extended Whitney trick}\label{subsec:whitney}

Let $(\conf,\iota,\phi)$ be a framed flow category with respect to $\dd=(d_k,\dots,d_{n+k})$. Suppose $\conf$ has objects $x, y$ with $|x|-|y|=m+1$ and write $(M,\iota,\phi)=(\cM(x,y),\iota_{x,y},\phi_{x,y})$. Suppose there is a framed cobordism rel.\ boundary $(W,\tilde{\iota},\tilde{\phi})$ between $(M,\iota,\phi)$ and some other embedded, framed $m$--dimensional $\langle m\rangle$--manifold $(M',\iota',\phi')$. We will define a new framed flow category $(\conf_W,\iota_W,\phi_W)$ such that there is a homotopy equivalence $\mathcal{X}(\conf)\simeq\mathcal{X}(\conf_W)$.

\begin{definition}\label{def:extW}With $(\conf,\iota,\phi)$ as above we will define $(\conf_W,\iota_W,\phi_W)$. First, we define the object set $\ob(\conf_W)=\{\bar{a}\,|\,a\in \ob(\conf)\}$. The moduli spaces of $\conf_W$ are given as follows:\begin{enumerate}
\item $\cM(\bar{x},\bar{y})=M'$.

\item If $a\in\ob(\conf)$ with $\cM(a,x)\neq \emptyset$, then the moduli space $\cM(\bar{a},\bar{y})$ is essentially given by glueing $W\times\cM(a,x)$ to $\cM(a,y)$ along the embedded boundary component\[M\times \cM(a,x)\subset \partial_{m+1}\cM(a,y).\]In order to ensure the smooth structures can be modified to match up when glueing, take small collar neighbourhoods $M\times\cM(a,x)\times[0,\eta)$ and $M\times\cM(a,x)\times(-\eta,0]$ inside $\cM(a,y)$ and $W\times\cM(a,x)$ respectively. Overlay the collars by setting $(p, q, t)\sim(p,q,(1-t)/2)$ for $t\in(0,\eta/2)$. Then define\[\begin{array}{l}\cM(\bar{a},\bar{y}):=((\cM(a,y)\sm (M\times\cM(a,x))\\
\qquad\qquad\qquad\qquad\qquad\qquad\sqcup (W\times \cM(a,x)\sm (M\times\cM(a,x))))/\sim.\end{array}\]

\item If $b\in\ob(\conf)$ with $\cM(y,b)\neq \emptyset$, then again take collar neighbourhoods $\cM(y,b)\times M\times[0,\eta)$ and $\cM(y,b)\times M\times(-\eta,0]$, now of \[\cM(y,b)\times M\subset \partial_{|y|-|b|+1}\cM(x,b)\] and $M\times \cM(y,b)\subset W\times\cM(y,b)$ respectively. Similarly to the overlay before, use $(q, p, t)\sim(q,p,(1-t)/2)$ for $t\in(0,\eta/2)$ to define\[\begin{array}{l}\cM(\bar{x},\bar{b}):=((\cM(x,b)\sm (M\times\cM(y,b))\\
\qquad\qquad\qquad\qquad\qquad\qquad\sqcup (W\times \cM(y,b)\sm (M\times\cM(y,b))))/\sim.\end{array}\]

\item If $a,b\in\ob(\conf)$ with both $\cM(a,x), \cM(y,b)\neq \emptyset$, then the following is a subset of the boundary $\partial \cM(a,b)$: \[\cM(y,b)\times M\times\cM(a,x)\subset \partial_{|y|-|b|+1}\cM(a,b)\cap\partial_{|x|-|b|+1}\cM(a,b).\]Take two 1--dimensional collars neighbourhoods of $\cM(y,b)\times M\times\cM(a,x)$. The first collar is taken inside $\partial_{|x|-|b|+1}\cM(a,b)$, and perpendicular to the boundary $\partial_{|y|-|b|+1}\cM(a,b)$, the second collar is taken \emph{vice versa}. Note that the product of the collars then defines a standard open neighbourhood inside the full moduli space $\cM(a,b)$:\begin{equation}\label{eq:nhood}\cM(y,b)\times M\times\cM(a,x)\times[0,2\eta)\times[0,2\eta)\subset \cM(a,b).\end{equation}In order to push $\cM(y,b)\times M\times\cM(a,x)$ out of the corner and replace it with $\cM(y,b)\times M'\times\cM(a,x)$, we will apply Lemma \ref{lem:smooth} to the standard open neighbourhood in line (\ref{eq:nhood}). Precisely, the embedded cobordism rel.\ boundary we will use as the input for the Lemma \ref{lem:smooth} is $(\cM(y,b)\times W\times\cM(a,x),\iota_{y,b}\times\tilde{\iota}\times\iota_{a,x})$ and the space $(X,\iota_X)$ is $(\cM(a,b),\iota_{a,b})$. Now identify the two collar directions of $\cM(y,b)\times M\times\cM(a,x)$ with axes of $\R_+\times\R_+$ so that we have $U_{2\eta}\subset [0,2\eta)\times [0,2\eta)$, the open ball of radius $2\eta$ from Lemma \ref{lem:smooth}. Corresponding to our inputs $X$ and $\cM(y,b)\times W\times\cM(a,x)$, there are also spaces $X'$ and $Y$ defined by Lemma \ref{lem:smooth}. Hence we define \[\cM(\bar{a},\bar{b}):=X'\sqcup Y/\sim,\]where $\sim$ is the equivalence relation defined in Lemma \ref{lem:smooth}. Note that this construction agrees with the constructions in (2) and (3), above, when considering the induced moduli spaces $\cM(\bar{x},\bar{b})$ and $\cM(\bar{a},\bar{y})$ in the boundary of $\cM(\bar{a},\bar{b})$.

\item In all other cases define $\cM(\bar{a},\bar{b})=\cM(a,b)$. Note this includes the case of $c\in\ob({\conf})$ with $\cM(x,c)\neq\emptyset\neq\cM(c,y)$.


\end{enumerate}

We now describe the embedding and framing $(\iota_W,\phi_W)$. For the case (1), we have $((\iota_W)_{\bar{x},\bar{y}}, (\phi_W)_{\bar{x},\bar{y}}):=(\iota',\phi')$, and for the case (5) $((\iota_W)_{\bar{a},\bar{b}}, (\phi_W)_{\bar{a},\bar{b}}):=(\iota_{{a},{b}}, \phi_{{a},{b}})$. In cases (2), (3) and (4), the embeddings and framings differ from those of $(\conf,\iota,\phi)$ only in a small neighbourhood of the glueing regions. The difference is determined by the embedding and framing of $Z$ which we defined in Lemma \ref{lem:smooth}.

\end{definition}

Finally we deduce the main result of this section.

\begin{proof}[Proof of Theorem \ref{thm:whitney}]First, we will assume $\dd$ has been enlarged enough that the framed cobordism rel.\ boundary $(W,\tilde{\iota},\tilde{\phi})$ is an embedding in $\E^\dd[|x|:|y|]\times[0,1]$. This does not affect the eventual stable homotopy type of the framed flow category.

Now, for each $a\in\ob(\conf)$, we will define a continuous map\[F_a:[0,1]\times\partial\mathcal{C}(a)\to X^{|a|-1};\qquad (t,p)\mapsto F_{a,t}(p),\]where $X^i$ is defined inductively for increasing $i$ by setting $X^0=\{pt\}$,\[X^i=X^{i-1}\cup_{F_a}\left([0,1]\times\mathcal{C}(a)\right),\]and with the union taken over all $a$ such that $|a|=i$. Furthermore, the maps $F_a$ will be defined such that $F_{a,0}$ and $F_{a,1}$ are the attaching maps for cells $\mathcal{C}(a)$ and $\mathcal{C}(\bar{a})$ in the CW-complexes $|\conf|$ and $|\conf_W|$ respectively. As such, the space $X:=\cup_iX^i$ can easily be seen to deformation retract onto each of $|\conf|$ and $|\conf_W|$, so that they are homotopy equivalent to one another.

It remains to define the $F_a$. Recall that $\cM(a,b)\times\mathcal{C}(b)$ is embedded as \[\begin{array}{l}\mathcal{C}_b(a)=[0,R]\times[-R,R]^{d_B}\times\dots\times[0,R]\times[-R,R]^{|b|-1}\times\{0\}\times\mathcal{C}_{b,1}\\ \qquad\qquad\qquad\qquad\qquad \times\{0\}\times[-\epsilon,\epsilon]^{|a|}\times\dots\times\{0\}\times [-\epsilon,\epsilon]^{d_{A-1}}\subset\partial\mathcal{C}(a),\end{array}\]where $\mathcal{C}_{b,1}$ is the subset of $\E^\dd[|a|:|b|]$ given by the framed embedding of $\cM(a,b)$. The cell attaching map for $\mathcal{C}(a)$ is given on $\mathcal{C}_b(a)\cong \cM(a,b)\times\mathcal{C}(b)$ by the projection to $\mathcal{C}(b)$; and on $\partial\mathcal{C}(a)\setminus\bigcup_b\mathcal{C}_{b}(a)$ by mapping to the basepoint. To build $F_a$ we must take the framed embedding of each moduli space $\cM(a,b)$ in turn, and `deform' the framed embedding to the framed embedding of $\cM(\bar{a},\bar{b})$. Note that we do not need each time slice of the deformation to look precisely like a framed embedding whose exterior can be collapsed (indeed in general it will not). To show continuity of $F_a$ we will only need that the \emph{track} of the deformation has a framing. Precisely, for each $a,b\in\ob(\conf)$ we will define a framed embedded cobordism between $\mathcal{C}_{b,1}$ and $\mathcal{C}_{\bar{b},1}$, giving a subset $\mathcal{C}_{b,W}\subset\E^\dd[|a|:|b|]\times[0,1]$ which in turn defines an embedding \[\begin{array}{l}[0,R]\times[-R,R]^{d_B}\times\dots\times[0,R]\times[-R,R]^{|b|-1}\times\{0\}\times\mathcal{C}_{b,W}\\ \qquad\qquad\qquad\qquad \times\{0\}\times[-\epsilon,\epsilon]^{|a|}\times\dots\times\{0\}\times [-\epsilon,\epsilon]^{d_{A-1}}\subset[0,1]\times\partial\mathcal{C}(a),\end{array}\]which is homeomorphic to $Z_{a,b}\times\mathcal{C}(b)$ (for some space $Z_{a,b}$, to be defined below). We can then define $F_a:[0,1]\times\partial\mathcal{C}(a)\to X^{|a|-1}$ as follows. On $Z_{a,b}\times\mathcal{C}(b)$, it is the projection to $\mathcal{C}(b)$; and on $[0,1]\times\partial\mathcal{C}(a)\setminus\bigcup_b\left(Z_{a,b}\times\mathcal{C}(b)\right)$, we map to the basepoint.

So we proceed to modify $\cM(a,b)$ in a series of case analyses based on (1), (2), (3), (4), of Definition \ref{def:extW}. Needless to say, the case of (5) requires no modification of $\cM(a,b)$. In this case, define a subspace $\mathcal{C}_{b,W}=\mathcal{C}_{b,1}\times[0,1]\subset\E^\dd[|a|:|b|]\times[0,1]$ and hence an embedding $Z_{a,b}\times\mathcal{C}(b)\subset \partial \mathcal{C}(a)$ where $Z_{a,b}$ is simply $\cM(a,b)\times[0,1]$.

For the first nontrivial case, set $a=x$ and $b=y$. Consider the embedding $\iota\times\id: M\times[0,1]\hookrightarrow \E^\dd[|x|:|y|]\times[0,1]$. Now use Lemma \ref{lem:smooth} and the framed cobordism $W$ to push $M$ out of the codimension 1 corner $\E^\dd[|x|:|y|]\times\{1\}$ and replace it with $M'$. This determines a subset $\mathcal{C}_{b,W}$ of $[-R,R]^{|y|}\times[0,R]\times\dots\times[0,R]\times[-R,R]^{|x|-1}\times[0,1]$. Moreover, this results in an embedding of $Z_{a,b}\times\mathcal{C}(y)$ in $[0,1]\times\partial\mathcal{C}(x)$ (here $Z_{a,b}:=Z$ refers to the result of applying Lemma \ref{lem:smooth}).

For the second case, set $a=x$, $b\neq y$ and $\cM(x,b)\neq\emptyset$. Then we have \[\cM(y,b)\times M\subset\partial\cM(x,b)\subset \E^\dd[|y|:|b|]\times\{0\}\times\E^\dd[|x|:|y|]\] and in the normal direction to this boundary we will take a $2\eta$ open collar neighbourhood. Consider that the product \[\cM(y,b)\times M\times[0,2\eta)\times[0,1]\subset\cM(x,b)\times[0,1]\]has a copy of $\cM(y,b)\times M$ in the corner $\left(\E^\dd[|y|:|b|]\times\{0\}\times\E^\dd[|x|:|y|]\right)\times\{0\}\times\{1\}$. Using Lemma \ref{lem:smooth} and the framed cobordism $\cM(y,b)\times W$, we push $\cM(y,b)\times M$ out of the corner and replace it with $\cM(y,b)\times M'$. Note that all embeddings were framed, so that resultant $Z_{a,b}$, is a framed embedding determining a subset $\mathcal{C}_{b,W}$ of $[-R,R]^{|b|}\times[0,R]\times\dots\times[0,R]\times[-R,R]^{|x|-1}\times[0,1]$ and hence an embedding of $Z_{a,b}\times\mathcal{C}(b)$ in $[0,1]\times\partial\mathcal{C}(x)$. Note that this agrees with the construction from the first case:\[\mathcal{C}_{y,W}\cap[-R,R]^{|y|}\times[0,R]\times\dots\times[0,R]\times[-R,R]^{|x|-1}\times[0,1]= \mathcal{C}_{b,W},\]as each case used Lemma \ref{lem:smooth}.

We must now look at cells $a\neq x$. There are two cases to consider. First suppose $a\neq x$, $b=y$ and that $\cM(a,y)\neq\emptyset$ and we wish to deform the embedding of $M\times\cM(a,x)$ appropriately. But this case proceeds identically to the case above where $a=x$, $b\neq y$ and $\cM(x,b)\neq\emptyset$, so we omit the details.

Finally, there is the case that $a\neq x$, $b\neq y$ and that $\cM(a,y)\neq\emptyset\neq\cM(x,b)$. This is treated by the same pushing out technique as before (but now we need to do it in codimension 3). Precisely, we must consider $\cM(y,b)\times M\times\cM(a,x)\subset \partial \cM(a,b)$. But observe there is a copy of $\cM(y,b)\times M\times\cM(a,x)$ in the codimension 3 corner $(0,0,1)\in[0,2\eta)\times[0,2\eta)\times[0,1]$ of the cylinder on the 2-way collar neighbourhood \[\cM(y,b)\times M\times\cM(a,x)\times[0,2\eta)\times[0,2\eta)\times[0,1]\subset \cM(a,b)\times[0,1].\]We push $\cM(y,b)\times M\times\cM(a,x)$ out of the corner using $\cM(y,b)\times W\times\cM(a,x)$ and we obtain $\mathcal{C}_{b,W}\cong Z_{a,b}\times\mathcal{C}(b)$ as in previous cases. This completes the case analysis.
\end{proof}

\begin{example}We illustrate a homotopy from the proof Theorem \ref{thm:whitney} in the case where $M, M'$ are 0--dimensional. Let $M$ be a point, $M'$ be 3 points and $W$ be the 2--component cobordism from $M$ to $M'$ as in Figure \ref{pic:cob}. We suppose $M$ is actually a corner of a larger space $N$, so that after the homotopy, $M'$ is in the corner of this larger space. This is shown in Figure \ref{pic:whit}, where we have drawn $N$ extending in the vertical axis, and drawn the framed cobordism $Z$ constructed in Theorem \ref{thm:whitney} by pushing $M$ out of the corner of $N\times\{1\}$ of $N\times[0,1]$ using $W$.

\begin{figure}
\begin{tikzpicture}
\draw (0,0) -- (5,0);
\draw (0,3) -- (5,3);
\draw (1,0) arc [radius=1, start angle=180, end angle= 0];
\draw (4,0) -- (4,3);
\draw [fill] (1,0) circle [radius=0.05];
\draw [fill] (3,0) circle [radius=0.05];
\draw [fill] (4,0) circle [radius=0.05];
\draw [fill] (4,3) circle [radius=0.05];
\node at (4.1,3.3) {M};
\node at (5.35,3.15) {$\E^\dd$};
\node at (5.35,0.15) {$\E^\dd$};
\node at (2.7,1.8) {$W$};
\draw [->] (2.9,1.7) -- (3.8,1.4);
\draw [->] (2.6,1.55) -- (2.5,1.1);
\node at (2.8,-0.5) {$M'$};
\draw [decorate,decoration={brace,amplitude=5pt,raise=8pt}] (0,0) -- (0,3);
\node [rotate=90] at (-0.9,1.5) {$\E^\dd\times[0,1]$};
\end{tikzpicture}
\caption{The cobordism $W$ between $M$ and $M'$}\label{pic:cob}
\end{figure}
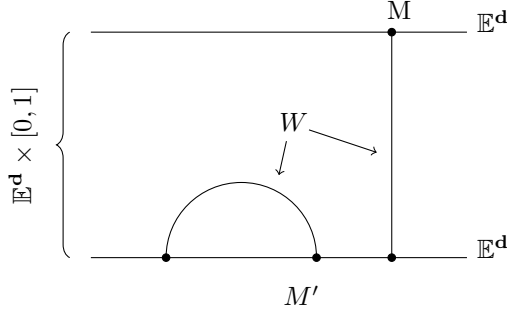

\begin{figure}[htb]
    \centering
    \includegraphics[width=0.8\textwidth]{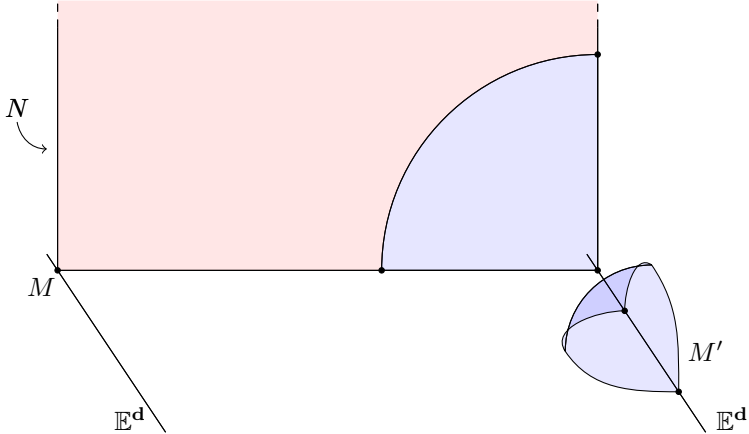}
\caption{The cobordism $Z$}\label{pic:whit}
  \end{figure}

\end{example}

\section{The disjoint union of three trefoils}
\label{sec:tref}

We make a sample calculation of a stable homotopy type using our framed flow category moves.  The example framed flow category $\cC_1$ that we shall start with is depicted in Figure \ref{fig:smashofMoore}.  We shall show explicitly that $\cC_1$ is equivalent via our flow category moves to the framed flow category $\cC_5$ depicted in Figure \ref{fig:smashofMooreIV}.  The form of $\cC_5$ corresponds to the Baues--Hennes classification of low homological width stable homotopy types and this is discussed further towards the end of the section.

We choose to consider the framed flow category $\cC_1$ since it arises in a particular context, namely as a framed flow category associated to the disjoint union of three trefoils by the techniques of \cite{JLS}.

\begin{definition}
	\label{defn:cC1}
There are eight objects in the category $\cC_1$ depicted in Figure \ref{fig:smashofMoore}, each labelled by a 3--tuple.  Any non-empty $0$--dimensional moduli space consists of two points each, as shown.  
A letter `p' indicates that the point is framed positively, while a letter `m' indicates a negative framing.

The $1$--dimensional moduli spaces are then given by four intervals each, with all of them framed $0$.  
Each such interval of course has two endpoints, and these are given as follows.

\begin{tikzpicture}[scale=0.8]
\node at (0,0) {$\cM(333,223)=$};

\node[label=below:$P_0\mathscr{P}_0$, label=above:{\color{blue}$233$}] (5_5_1)at (1.9,0) {};
\node[label=below:$\tilde{P}_0\mathscr{M}_0$, label=above:{\color{blue}$323$}] (5_5_2)at (3.7,0) {};

\draw[shorten >=-0.1cm,shorten <=-0.1cm,|-|] (5_5_1)--(5_5_2) node[sloped, above,pos=0.5, red]{$0$};

\node[label=below:$P_0\mathscr{P}_1$, label=above:{\color{blue}$233$}] (5_5_3)at (5.2,0) {};
\node[label=below:$\tilde{P}_1\mathscr{M}_0$, label=above:{\color{blue}$323$}] (5_5_4)at (7,0) {};

\draw[shorten >=-0.1cm,shorten <=-0.1cm,|-|] (5_5_3)--(5_5_4) node[sloped, above,pos=0.5, red]{$0$};

\node[label=below:$P_1\mathscr{P}_0$, label=above:{\color{blue}$233$}] (5_5_5)at (8.5,0) {};
\node[label=below:$\tilde{P}_0\mathscr{M}_1$, label=above:{\color{blue}$323$}] (5_5_6)at (10.3,0) {};

\draw[shorten >=-0.1cm,shorten <=-0.1cm,|-|] (5_5_5)--(5_5_6) node[sloped, above,pos=0.5, red]{$0$};

\node[label=below:$P_1\mathscr{P}_1$, label=above:{\color{blue}$233$}] (5_5_7)at (11.8,0) {};
\node[label=below:$\tilde{P}_1\mathscr{M}_1$, label=above:{\color{blue}$323$}] (5_5_8)at (13.6,0) {};

\draw[shorten >=-0.1cm,shorten <=-0.1cm,|-|] (5_5_7)--(5_5_8) node[sloped, above,pos=0.5, red]{$0$};

\node at (0,-2) {$\cM(333,232)=$};

\node[label=below:$M_0\mathscr{P}_0$, label=above:{\color{blue}$233$}] (5_5_1)at (1.9,-2) {};
\node[label=below:$\bar{P}_0\tilde{\mathscr{P}}_0$, label=above:{\color{blue}$332$}] (5_5_2)at (3.7,-2) {};

\draw[shorten >=-0.1cm,shorten <=-0.1cm,|-|] (5_5_1)--(5_5_2) node[sloped, above,pos=0.5, red]{$0$};

\node[label=below:$M_0\mathscr{P}_1$, label=above:{\color{blue}$233$}] (5_5_3)at (5.2,-2) {};
\node[label=below:$\bar{P}_1\tilde{\mathscr{P}}_0$, label=above:{\color{blue}$332$}] (5_5_4)at (7,-2) {};

\draw[shorten >=-0.1cm,shorten <=-0.1cm,|-|] (5_5_3)--(5_5_4) node[sloped, above,pos=0.5, red]{$0$};

\node[label=below:$M_1\mathscr{P}_0$, label=above:{\color{blue}$233$}] (5_5_5)at (8.5,-2) {};
\node[label=below:$\bar{P}_0\tilde{\mathscr{P}}_1$, label=above:{\color{blue}$332$}] (5_5_6)at (10.3,-2) {};

\draw[shorten >=-0.1cm,shorten <=-0.1cm,|-|] (5_5_5)--(5_5_6) node[sloped, above,pos=0.5, red]{$0$};

\node[label=below:$M_1\mathscr{P}_1$, label=above:{\color{blue}$233$}] (5_5_7)at (11.8,-2) {};
\node[label=below:$\bar{P}_1\tilde{\mathscr{P}}_1$, label=above:{\color{blue}$332$}] (5_5_8)at (13.6,-2) {};

\draw[shorten >=-0.1cm,shorten <=-0.1cm,|-|] (5_5_7)--(5_5_8) node[sloped, above,pos=0.5, red]{$0$};

\node at (0,-4) {$\cM(333,322)=$};

\node[label=below:$\tilde{M}_0\mathscr{M}_0$, label=above:{\color{blue}$323$}] (5_5_1)at (1.9,-4) {};
\node[label=below:$\bar{M}_0\tilde{\mathscr{P}}_0$, label=above:{\color{blue}$332$}] (5_5_2)at (3.7,-4) {};

\draw[shorten >=-0.1cm,shorten <=-0.1cm,|-|] (5_5_1)--(5_5_2) node[sloped, above,pos=0.5, red]{$0$};

\node[label=below:$\tilde{M}_0\mathscr{M}_1$, label=above:{\color{blue}$323$}] (5_5_3)at (5.2,-4) {};
\node[label=below:$\bar{M}_1\tilde{\mathscr{P}}_0$, label=above:{\color{blue}$332$}] (5_5_4)at (7,-4) {};

\draw[shorten >=-0.1cm,shorten <=-0.1cm,|-|] (5_5_3)--(5_5_4) node[sloped, above,pos=0.5, red]{$0$};

\node[label=below:$\tilde{M}_1\mathscr{M}_0$, label=above:{\color{blue}$323$}] (5_5_5)at (8.5,-4) {};
\node[label=below:$\bar{M}_0\tilde{\mathscr{P}}_1$, label=above:{\color{blue}$332$}] (5_5_6)at (10.3,-4) {};

\draw[shorten >=-0.1cm,shorten <=-0.1cm,|-|] (5_5_5)--(5_5_6) node[sloped, above,pos=0.5, red]{$0$};

\node[label=below:$\tilde{M}_1\mathscr{M}_1$, label=above:{\color{blue}$323$}] (5_5_7)at (11.8,-4) {};
\node[label=below:$\bar{M}_1\tilde{\mathscr{P}}_1$, label=above:{\color{blue}$332$}] (5_5_8)at (13.6,-4) {};

\draw[shorten >=-0.1cm,shorten <=-0.1cm,|-|] (5_5_7)--(5_5_8) node[sloped, above,pos=0.5, red]{$0$};

\node at (0,-6) {$\cM(233,222)=$};

\node[label=below:$p_0P_0$, label=above:{\color{blue}$223$}] (5_5_1)at (1.9,-6) {};
\node[label=below:$\tilde{p}_0M_0$, label=above:{\color{blue}$232$}] (5_5_2)at (3.7,-6) {};

\draw[shorten >=-0.1cm,shorten <=-0.1cm,|-|] (5_5_1)--(5_5_2) node[sloped, above,pos=0.5, red]{$0$};

\node[label=below:$p_0P_1$, label=above:{\color{blue}$223$}] (5_5_3)at (5.2,-6) {};
\node[label=below:$\tilde{p}_1M_0$, label=above:{\color{blue}$232$}] (5_5_4)at (7,-6) {};

\draw[shorten >=-0.1cm,shorten <=-0.1cm,|-|] (5_5_3)--(5_5_4) node[sloped, above,pos=0.5, red]{$0$};

\node[label=below:$p_1P_0$, label=above:{\color{blue}$223$}] (5_5_5)at (8.5,-6) {};
\node[label=below:$\tilde{p}_0M_1$, label=above:{\color{blue}$232$}] (5_5_6)at (10.3,-6) {};

\draw[shorten >=-0.1cm,shorten <=-0.1cm,|-|] (5_5_5)--(5_5_6) node[sloped, above,pos=0.5, red]{$0$};

\node[label=below:$p_1P_1$, label=above:{\color{blue}$223$}] (5_5_7)at (11.8,-6) {};
\node[label=below:$\tilde{p}_1M_1$, label=above:{\color{blue}$232$}] (5_5_8)at (13.6,-6) {};

\draw[shorten >=-0.1cm,shorten <=-0.1cm,|-|] (5_5_7)--(5_5_8) node[sloped, above,pos=0.5, red]{$0$};

\node at (0,-8) {$\cM(323,222)=$};

\node[label=below:$p_0\tilde{P}_0$, label=above:{\color{blue}$223$}] (5_5_1)at (1.9,-8) {};
\node[label=below:$\bar{p}_0\tilde{M}_0$, label=above:{\color{blue}$322$}] (5_5_2)at (3.7,-8) {};

\draw[shorten >=-0.1cm,shorten <=-0.1cm,|-|] (5_5_1)--(5_5_2) node[sloped, above,pos=0.5, red]{$0$};

\node[label=below:$p_0\tilde{P}_1$, label=above:{\color{blue}$223$}] (5_5_3)at (5.2,-8) {};
\node[label=below:$\bar{p}_1\tilde{M}_0$, label=above:{\color{blue}$322$}] (5_5_4)at (7,-8) {};

\draw[shorten >=-0.1cm,shorten <=-0.1cm,|-|] (5_5_3)--(5_5_4) node[sloped, above,pos=0.5, red]{$0$};

\node[label=below:$p_1\tilde{P}_0$, label=above:{\color{blue}$223$}] (5_5_5)at (8.5,-8) {};
\node[label=below:$\bar{p}_0\tilde{M}_1$, label=above:{\color{blue}$322$}] (5_5_6)at (10.3,-8) {};

\draw[shorten >=-0.1cm,shorten <=-0.1cm,|-|] (5_5_5)--(5_5_6) node[sloped, above,pos=0.5, red]{$0$};

\node[label=below:$p_1\tilde{P}_1$, label=above:{\color{blue}$223$}] (5_5_7)at (11.8,-8) {};
\node[label=below:$\bar{p}_1\tilde{M}_1$, label=above:{\color{blue}$322$}] (5_5_8)at (13.6,-8) {};

\draw[shorten >=-0.1cm,shorten <=-0.1cm,|-|] (5_5_7)--(5_5_8) node[sloped, above,pos=0.5, red]{$0$};

\end{tikzpicture}

\begin{tikzpicture}[scale=0.8]

\node at (0,-4) {$\cM(332,222)=$};

\node[label=below:$\tilde{p}_0\bar{P}_0$, label=above:{\color{blue}$232$}] (5_5_1)at (1.9,-4) {};
\node[label=below:$\bar{p}_0\bar{M}_0$, label=above:{\color{blue}$322$}] (5_5_2)at (3.7,-4) {};

\draw[shorten >=-0.1cm,shorten <=-0.1cm,|-|] (5_5_1)--(5_5_2) node[sloped, above,pos=0.5, red]{$0$};

\node[label=below:$\tilde{p}_0\bar{P}_1$, label=above:{\color{blue}$232$}] (5_5_3)at (5.2,-4) {};
\node[label=below:$\bar{p}_1\bar{M}_0$, label=above:{\color{blue}$322$}] (5_5_4)at (7,-4) {};

\draw[shorten >=-0.1cm,shorten <=-0.1cm,|-|] (5_5_3)--(5_5_4) node[sloped, above,pos=0.5, red]{$0$};

\node[label=below:$\tilde{p}_1\bar{P}_0$, label=above:{\color{blue}$232$}] (5_5_5)at (8.5,-4) {};
\node[label=below:$\bar{p}_0\bar{M}_1$, label=above:{\color{blue}$322$}] (5_5_6)at (10.3,-4) {};

\draw[shorten >=-0.1cm,shorten <=-0.1cm,|-|] (5_5_5)--(5_5_6) node[sloped, above,pos=0.5, red]{$0$};

\node[label=below:$\tilde{p}_1\bar{P}_1$, label=above:{\color{blue}$232$}] (5_5_7)at (11.8,-4) {};
\node[label=below:$\bar{p}_1\bar{M}_1$, label=above:{\color{blue}$322$}] (5_5_8)at (13.6,-4) {};

\draw[shorten >=-0.1cm,shorten <=-0.1cm,|-|] (5_5_7)--(5_5_8) node[sloped, above,pos=0.5, red]{$0$};

\end{tikzpicture}

The $2$--dimensional moduli space $\cM(333,222)$ consists of four hexagons.   It will follow from the Baues--Hennes classification \cite{MR1113684} that the associated stable homotopy type is determined by the action of $\Sq^1$ and $\Sq^2$.  These operations are in turn determined by the framed moduli spaces of dimensions $0$ and $1$, so we will not keep track of the $2$-dimensional moduli space.
\end{definition}

For those less interested in the Khovanov stable homotopy type, the next proposition may be skipped.  We focus on quantum degree $q=21$ of the Khovanov stable homotopy type of the disjoint union of three right-handed trefoils, since we know that there is a non-trivial $\Sq^3$ in this degree.  Indeed, the general formula of \cite{LawLipSar} confirms the existence of the smash of Moore spaces $$M(\Z/2\Z,2)\wedge M(\Z/2\Z,2) \wedge M(\Z/2\Z,2)$$ as a wedge summand.

\begin{proposition}
	\label{prop:eg_arises_from_tref}
Let $L$ be the disjoint union of three right-handed trefoils.  The technique of \cite{JLS} constructs a framed flow category $\mathscr{L}^{\mathrm{Kh}}(L)$.  
In quantum degree $q=21$, this framed flow category is the disjoint union of the framed flow category $\cC_1$ described in Definition \ref{defn:cC1} with some other framed flow category.
\end{proposition}

\begin{proof}
The following calculation is similar to \cite[\S 4.3]{JLSmorse}. In homological degree $9$ we get exactly one object which is based at $(3,3,3)$. The smoothing of this object consists of three circles, each of which is decorated with a $-$. The objects of homological degree $8$ are based at $(2,3,3)$, $(3,2,3)$ and $(3,3,2)$. Again the smoothings are three circles, and one of them is decorated with a $+$. We therefore get $9$ objects of degree $8$. Note however that only three of those have non-empty moduli space with the object of degree $9$, namely those where the $+$ corresponds to the position of the $2$ in the base triple.

We will now only consider those objects that have non-empty moduli spaces with the object based at $(3,3,3)$, as these objects will give rise exactly to the product of Moore spaces predicted above. We then get three objects of homological degree $7$, based at $(2,2,3)$, $(2,3,2)$ and $(3,2,2)$, where the circle corresponding to the $3$ is decorated $-$, and one object of homological degree $6$ based at $(2,2,2)$ with all three circles decorated $+$. It is easy to see that these objects do not share non-empty moduli spaces with other objects in $\mathscr{L}^{\mathrm{Kh}}(L)$.
\end{proof}

\begin{figure}[ht]
\begin{tikzpicture}[scale=0.85]
\node at (0,6) {$333$};
\node at (-2,4) {$233$};
\node at (0,4) {$323$};
\node at (2,4) {$332$};
\node at (-2,2) {$223$};
\node at (0,2) {$232$};
\node at (2,2) {$322$};
\node at (0,0) {$222$};

\draw [shorten >=0.3cm,shorten <=0.3cm,->] (0,6) -- node [above,sloped,scale=0.7] {$\mathscr{P}_0\mathscr{P}_1$} (-2,4);
\draw [shorten >=0.3cm,shorten <=0.3cm,->] (0,6) -- node [above,sloped,scale=0.7] {$\mathscr{M}_0\mathscr{M}_1$} (0,4);
\draw [shorten >=0.3cm,shorten <=0.3cm,->] (0,6) -- node [above,sloped,scale=0.7] {$\tilde{\mathscr{P}}_0\tilde{\mathscr{P}}_1$} (2,4);
\draw [shorten >=0.3cm,shorten <=0.3cm,->] (-2,4) -- node [above,sloped,scale=0.7] {$P_0P_1$} (-2,2);
\draw [shorten >=0.3cm,shorten <=0.3cm,->] (-2,4) -- node [above,sloped,scale=0.7,pos=0.25] {$M_0M_1$} (0,2);
\draw [shorten >=0.3cm,shorten <=0.3cm,->] (0,4) -- node [above,sloped,scale=0.7,pos=0.25] {$\tilde{P}_0\tilde{P}_1$} (-2,2);
\draw [shorten >=0.3cm,shorten <=0.3cm,->] (0,4) -- node [above,sloped,scale=0.7,pos=0.25] {$\tilde{M}_0\tilde{M}_1$} (2,2);
\draw [shorten >=0.3cm,shorten <=0.3cm,->] (2,4) -- node [above,sloped,scale=0.7,pos=0.25] {$\bar{P}_0\bar{P}_1$} (0,2);
\draw [shorten >=0.3cm,shorten <=0.3cm,->] (2,4) -- node [above,sloped,scale=0.7] {$\bar{M}_0\bar{M}_1$} (2,2);
\draw [shorten >=0.3cm,shorten <=0.3cm,->] (-2,2) -- node [above,sloped,scale=0.7] {$p_0p_1$} (0,0);
\draw [shorten >=0.3cm,shorten <=0.3cm,->] (0,2) -- node [above,sloped,scale=0.7] {$\tilde{p}_0\tilde{p}_1$} (0,0);
\draw [shorten >=0.3cm,shorten <=0.3cm,->] (2,2) -- node [above,sloped,scale=0.7] {$\bar{p}_0\bar{p}_1$} (0,0);
\end{tikzpicture}
\caption{A subcategory $\cC_1$ of $\mathscr{L}^{\mathrm{Kh}}(L)$.}
\label{fig:smashofMoore}
\end{figure}

We are now going to perform handle slides and Whitney tricks, simplifying the framed flow category $\cC_1$ in a sequence of four propositions until we arrive at the `Baues--Hennes' category $\cC_5$.

\begin{proposition}
	\label{prop:cC1_to_cC2}
	The framed flow category $\cC_1$ is move-equivalent to a framed flow category $\cC_2$, depicted in Figure \ref{fig:smashofMooreI}.  In this figure the $1$--dimensional moduli space denoted $\eta$ is a non-trivially framed circle.
\end{proposition}

\begin{proof}
To begin, we slide $223$ over $232$ via a $(-)$--handle slide.  We will write an overline above all objects to indicate the resulting flow category, but these overlines will disappear again in time for the next move. The affected moduli spaces are
\[
 \cM(\overline{333},\overline{232})=\cM(333,232) \sqcup \cM(333,223)
\]
with no change in framings by Proposition \ref{prop:framechanges}.1(a). Furthermore, both $\cM(\overline{233},\overline{222})$ and $\cM(\overline{323},\overline{222})$ get four new intervals, each framed $0$ by Proposition \ref{prop:framechanges}.2(a). Note these new intervals correspond to $\cM(232,222)\times \cM(a,223)$ with $a=233$ or $a=323$. The points $\tilde{p}_0,\tilde{p}_1\in \cM(232,222)$ create new points $m_0,m_1\in \cM(\overline{223},\overline{222})$, $P_0,P_1\in \cM(233,223)$ create $\hat{P}_0,\hat{P}_1\in \cM(\overline{233},\overline{232})$, and $\tilde{P}_0,\tilde{P}_1\in \cM(323,223)$ create $\check{P}_0,\check{P}_1\in \cM(\overline{323},\overline{232})$.

The new intervals in $\cM(\overline{233},\overline{222})$ are given by

\begin{tikzpicture}
\node[label=below:$m_0P_0$, label=above:{\color{blue}$\overline{223}$}] (5_5_1)at (0,0) {};
\node[label=below:$\tilde{p}_0\hat{P}_0$, label=above:{\color{blue}$\overline{232}$}] (5_5_2)at (1.6,0) {};

\draw[shorten >=-0.1cm,shorten <=-0.1cm,|-|] (5_5_1)--(5_5_2) node[sloped, above,pos=0.5, red]{$0$};

\node[label=below:$m_0P_1$, label=above:{\color{blue}$\overline{223}$}] (5_5_3)at (3.3,0) {};
\node[label=below:$\tilde{p}_0\hat{P}_1$, label=above:{\color{blue}$\overline{232}$}] (5_5_4)at (4.9,0) {};

\draw[shorten >=-0.1cm,shorten <=-0.1cm,|-|] (5_5_3)--(5_5_4) node[sloped, above,pos=0.5, red]{$0$};

\node[label=below:$m_1P_0$, label=above:{\color{blue}$\overline{223}$}] (5_5_5)at (6.6,0) {};
\node[label=below:$\tilde{p}_1\hat{P}_0$, label=above:{\color{blue}$\overline{232}$}] (5_5_6)at (8.2,0) {};

\draw[shorten >=-0.1cm,shorten <=-0.1cm,|-|] (5_5_5)--(5_5_6) node[sloped, above,pos=0.5, red]{$0$};

\node[label=below:$m_1P_1$, label=above:{\color{blue}$\overline{223}$}] (5_5_7)at (9.9,0) {};
\node[label=below:$\tilde{p}_1\hat{P}_1$, label=above:{\color{blue}$\overline{232}$}] (5_5_8)at (11.5,0) {};

\draw[shorten >=-0.1cm,shorten <=-0.1cm,|-|] (5_5_7)--(5_5_8) node[sloped, above,pos=0.5, red]{$0$};
\end{tikzpicture}

and similarly in $\cM(\overline{323},\overline{222})$ they are given by

\begin{tikzpicture}
\node[label=below:$m_0\tilde{P}_0$, label=above:{\color{blue}$\overline{223}$}] (5_5_1)at (0,0) {};
\node[label=below:$\tilde{p}_0\check{P}_0$, label=above:{\color{blue}$\overline{232}$}] (5_5_2)at (1.6,0) {};

\draw[shorten >=-0.1cm,shorten <=-0.1cm,|-|] (5_5_1)--(5_5_2) node[sloped, above,pos=0.5, red]{$0$};

\node[label=below:$m_0\tilde{P}_1$, label=above:{\color{blue}$\overline{223}$}] (5_5_3)at (3.3,0) {};
\node[label=below:$\tilde{p}_0\check{P}_1$, label=above:{\color{blue}$\overline{232}$}] (5_5_4)at (4.9,0) {};

\draw[shorten >=-0.1cm,shorten <=-0.1cm,|-|] (5_5_3)--(5_5_4) node[sloped, above,pos=0.5, red]{$0$};

\node[label=below:$m_1\tilde{P}_0$, label=above:{\color{blue}$\overline{223}$}] (5_5_5)at (6.6,0) {};
\node[label=below:$\tilde{p}_1\check{P}_0$, label=above:{\color{blue}$\overline{232}$}] (5_5_6)at (8.2,0) {};

\draw[shorten >=-0.1cm,shorten <=-0.1cm,|-|] (5_5_5)--(5_5_6) node[sloped, above,pos=0.5, red]{$0$};

\node[label=below:$m_1\tilde{P}_1$, label=above:{\color{blue}$\overline{223}$}] (5_5_7)at (9.9,0) {};
\node[label=below:$\tilde{p}_1\check{P}_1$, label=above:{\color{blue}$\overline{232}$}] (5_5_8)at (11.5,0) {};

\draw[shorten >=-0.1cm,shorten <=-0.1cm,|-|] (5_5_7)--(5_5_8) node[sloped, above,pos=0.5, red]{$0$};
\end{tikzpicture}

We now perform the Whitney trick in $\cM(\overline{223},\overline{222})$ with $p_0,m_0$ and with $p_1,m_1$.  The result is that in $\cM(\overline{233},\overline{222})$ and $\cM(\overline{323},\overline{222})$ intervals are glued together. For example, the endpoint $p_0P_0$ in the old $\cM(233,222)$ is identified with the endpoint $m_0P_0$ in one of the new intervals. In fact, in each case two intervals are glued together to form a new interval. Furthermore, by \cite[Proposition 3.3]{JLSmorse}, the framing value of each new interval is $1$. We then get

\begin{tikzpicture}[scale=0.8]
\node at (0,0) {$\cM(\overline{233},\overline{222})=$};

\node[label=below:$\tilde{p}_0\hat{P}_0$, label=above:{\color{blue}$\overline{232}$}] (5_5_1)at (1.9,0) {};
\node[label=below:$\tilde{p}_0M_0$, label=above:{\color{blue}$\overline{232}$}] (5_5_2)at (3.7,0) {};

\draw[shorten >=-0.1cm,shorten <=-0.1cm,|-|] (5_5_1)--(5_5_2) node[sloped, above,pos=0.5, red]{$1$};

\node[label=below:$\tilde{p}_0\hat{P}_1$, label=above:{\color{blue}$\overline{232}$}] (5_5_3)at (5.2,0) {};
\node[label=below:$\tilde{p}_1M_0$, label=above:{\color{blue}$\overline{232}$}] (5_5_4)at (7,0) {};

\draw[shorten >=-0.1cm,shorten <=-0.1cm,|-|] (5_5_3)--(5_5_4) node[sloped, above,pos=0.5, red]{$1$};

\node[label=below:$\tilde{p}_1\hat{P}_0$, label=above:{\color{blue}$\overline{232}$}] (5_5_5)at (8.5,0) {};
\node[label=below:$\tilde{p}_0M_1$, label=above:{\color{blue}$\overline{232}$}] (5_5_6)at (10.3,0) {};

\draw[shorten >=-0.1cm,shorten <=-0.1cm,|-|] (5_5_5)--(5_5_6) node[sloped, above,pos=0.5, red]{$1$};

\node[label=below:$\tilde{p}_1\hat{P}_1$, label=above:{\color{blue}$\overline{232}$}] (5_5_7)at (11.8,0) {};
\node[label=below:$\tilde{p}_1M_1$, label=above:{\color{blue}$\overline{232}$}] (5_5_8)at (13.6,0) {};

\draw[shorten >=-0.1cm,shorten <=-0.1cm,|-|] (5_5_7)--(5_5_8) node[sloped, above,pos=0.5, red]{$1$};

\node at (0,-2) {$\cM(\overline{323},\overline{222})=$};

\node[label=below:$\tilde{p}_0\check{P}_0$, label=above:{\color{blue}$\overline{232}$}] (5_5_1)at (1.9,-2) {};
\node[label=below:$\bar{p}_0\tilde{M}_0$, label=above:{\color{blue}$\overline{322}$}] (5_5_2)at (3.7,-2) {};

\draw[shorten >=-0.1cm,shorten <=-0.1cm,|-|] (5_5_1)--(5_5_2) node[sloped, above,pos=0.5, red]{$1$};

\node[label=below:$\tilde{p}_0\check{P}_1$, label=above:{\color{blue}$\overline{232}$}] (5_5_3)at (5.2,-2) {};
\node[label=below:$\bar{p}_1\tilde{M}_0$, label=above:{\color{blue}$\overline{322}$}] (5_5_4)at (7,-2) {};

\draw[shorten >=-0.1cm,shorten <=-0.1cm,|-|] (5_5_3)--(5_5_4) node[sloped, above,pos=0.5, red]{$1$};

\node[label=below:$\tilde{p}_1\check{P}_0$, label=above:{\color{blue}$\overline{232}$}] (5_5_5)at (8.5,-2) {};
\node[label=below:$\bar{p}_0\tilde{M}_1$, label=above:{\color{blue}$\overline{322}$}] (5_5_6)at (10.3,-2) {};

\draw[shorten >=-0.1cm,shorten <=-0.1cm,|-|] (5_5_5)--(5_5_6) node[sloped, above,pos=0.5, red]{$1$};

\node[label=below:$\tilde{p}_1\check{P}_1$, label=above:{\color{blue}$\overline{232}$}] (5_5_7)at (11.8,-2) {};
\node[label=below:$\bar{p}_1\tilde{M}_1$, label=above:{\color{blue}$\overline{322}$}] (5_5_8)at (13.6,-2) {};

\draw[shorten >=-0.1cm,shorten <=-0.1cm,|-|] (5_5_7)--(5_5_8) node[sloped, above,pos=0.5, red]{$1$};

\end{tikzpicture}

In the next step we again perform the Whitney trick, this time using $\hat{P}_0,M_0$ and $\hat{P}_1,M_1$ in $\cM(\overline{233},\overline{232})$.  Note that we now remove the overline from the objects. The effect on $\cM(333,232)$ is that the eight intervals turn into four similarly to the case above. More precisely, we get

\begin{tikzpicture}[scale=0.8]
\node at (0,-2) {$\cM(333,232)=$};

\node[label=below:$\check{P}_0\mathscr{M}_0$, label=above:{\color{blue}$323$}] (5_5_1)at (1.9,-2) {};
\node[label=below:$\bar{P}_0\tilde{\mathscr{P}}_0$, label=above:{\color{blue}$332$}] (5_5_2)at (3.7,-2) {};

\draw[shorten >=-0.1cm,shorten <=-0.1cm,|-|] (5_5_1)--(5_5_2) node[sloped, above,pos=0.5, red]{$1$};

\node[label=below:$\check{P}_1\mathscr{M}_0$, label=above:{\color{blue}$323$}] (5_5_3)at (5.2,-2) {};
\node[label=below:$\bar{P}_1\tilde{\mathscr{P}}_0$, label=above:{\color{blue}$332$}] (5_5_4)at (7,-2) {};

\draw[shorten >=-0.1cm,shorten <=-0.1cm,|-|] (5_5_3)--(5_5_4) node[sloped, above,pos=0.5, red]{$1$};

\node[label=below:$\check{P}_0\mathscr{M}_1$, label=above:{\color{blue}$323$}] (5_5_5)at (8.5,-2) {};
\node[label=below:$\bar{P}_0\tilde{\mathscr{P}}_1$, label=above:{\color{blue}$332$}] (5_5_6)at (10.3,-2) {};

\draw[shorten >=-0.1cm,shorten <=-0.1cm,|-|] (5_5_5)--(5_5_6) node[sloped, above,pos=0.5, red]{$1$};

\node[label=below:$\check{P}_1\mathscr{M}_1$, label=above:{\color{blue}$323$}] (5_5_7)at (11.8,-2) {};
\node[label=below:$\bar{P}_1\tilde{\mathscr{P}}_1$, label=above:{\color{blue}$332$}] (5_5_8)at (13.6,-2) {};

\draw[shorten >=-0.1cm,shorten <=-0.1cm,|-|] (5_5_7)--(5_5_8) node[sloped, above,pos=0.5, red]{$1$};
\end{tikzpicture}

The moduli space $\cM(233,222)$ turns into a closed manifold. In fact, the outer intervals result in one circle each, and the inner two intervals are glued together along their endpoints to form a single circle. By \cite[Proposition 3.4]{JLSmorse} all circles are labelled with $0$, which means that each circle is \emph{non-trivially} framed (compare \cite{JLSmorse} for framing conventions). Using the extended Whitney trick, we can reduce this to one non-trivially framed circle, which we denote by $\cM(233,222)=\eta$.

The result is the framed flow category $\cC_2$ depicted in Figure \ref{fig:smashofMooreI}.
\end{proof}

\begin{proposition}
	\label{prop:cC2_to_cC3}
	The framed flow category $\cC_2$ is move equivalent to the framed flow category $\cC_3$ depicted in Figure \ref{fig:smashofMooreII}.  In Figure \ref{fig:smashofMooreII} we denote non-trivially framed circles either by $\xi$ or $\eta$.
\end{proposition}

\begin{proof}
Starting from this category, we slide $322$ over $232$ with a $(-)$--handle slide. This introduces extra points $M_0,M_1$ in $\cM(\overline{332},\overline{232})$ and $\check{M}_0,\check{M}_1$ in $\cM(\overline{323},\overline{232})$, as well as $m_0,m_1$ in $\cM(\overline{322},\overline{222})$. Similar to the last slide in Proposition \ref{prop:cC1_to_cC2}, the moduli spaces $\cM(\overline{333},\overline{232})$, $\cM(\overline{323},\overline{222})$ and $\cM(\overline{332},\overline{222})$ each acquire four new intervals.

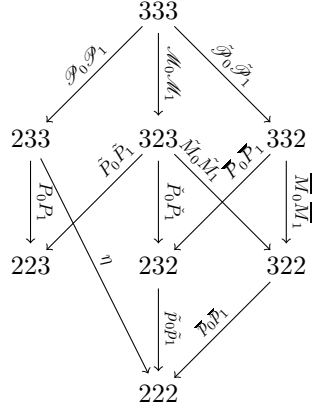
\begin{figure}[ht]
\begin{tikzpicture}[scale=0.85]
\node at (0,6) {$333$};
\node at (-2,4) {$233$};
\node at (0,4) {$323$};
\node at (2,4) {$332$};
\node at (-2,2) {$223$};
\node at (0,2) {$232$};
\node at (2,2) {$322$};
\node at (0,0) {$222$};

\draw [shorten >=0.3cm,shorten <=0.3cm,->] (0,6) -- node [above,sloped,scale=0.7] {$\mathscr{P}_0\mathscr{P}_1$} (-2,4);
\draw [shorten >=0.3cm,shorten <=0.3cm,->] (0,6) -- node [above,sloped,scale=0.7] {$\mathscr{M}_0\mathscr{M}_1$} (0,4);
\draw [shorten >=0.3cm,shorten <=0.3cm,->] (0,6) -- node [above,sloped,scale=0.7] {$\tilde{\mathscr{P}}_0\tilde{\mathscr{P}}_1$} (2,4);
\draw [shorten >=0.3cm,shorten <=0.3cm,->] (-2,4) -- node [above,sloped,scale=0.7] {$P_0P_1$} (-2,2);
\draw [shorten >=0.3cm,shorten <=0.3cm,->] (0,4) -- node [above,sloped,scale=0.7] {$\check{P}_0\check{P}_1$} (0,2);
\draw [shorten >=0.3cm,shorten <=0.3cm,->] (0,4) -- node [above,sloped,scale=0.7,pos=0.25] {$\tilde{P}_0\tilde{P}_1$} (-2,2);
\draw [shorten >=0.3cm,shorten <=0.3cm,->] (0,4) -- node [above,sloped,scale=0.7,pos=0.25] {$\tilde{M}_0\tilde{M}_1$} (2,2);
\draw [shorten >=0.3cm,shorten <=0.3cm,->] (2,4) -- node [above,sloped,scale=0.7,pos=0.25] {$\bar{P}_0\bar{P}_1$} (0,2);
\draw [shorten >=0.3cm,shorten <=0.3cm,->] (2,4) -- node [above,sloped,scale=0.7] {$\bar{M}_0\bar{M}_1$} (2,2);
\draw [shorten >=0.3cm,shorten <=0.3cm,->] (-2,4) -- node [above,sloped,scale=0.7] {$\eta$} (0,0);
\draw [shorten >=0.3cm,shorten <=0.3cm,->] (0,2) -- node [above,sloped,scale=0.7] {$\tilde{p}_0\tilde{p}_1$} (0,0);
\draw [shorten >=0.3cm,shorten <=0.3cm,->] (2,2) -- node [above,sloped,scale=0.7] {$\bar{p}_0\bar{p}_1$} (0,0);
\end{tikzpicture}
\caption{The framed flow category $\cC_2$.}
\label{fig:smashofMooreI}
\end{figure}

Performing all obvious extended Whitney tricks (0--dimensional and 1--dimensional) now leads to the flow category indicated in Figure \ref{fig:smashofMooreII}. One easily checks that $\cM(333,232)$, $\cM(323,222)$ and $\cM(332,222)$ all turn into three circles, and after extended Whitney tricks all contain one non-trivially framed circle. We denote the non-trivially framed circle in $\cM(333,232)$ by $\xi$, mainly to follow the conventions in \cite{MR1113684}.
\end{proof}

\begin{figure}[ht]
\begin{tikzpicture}[scale=0.8]
\node at (0,6) {$333$};
\node at (-2,4) {$233$};
\node at (0,4) {$323$};
\node at (2,4) {$332$};
\node at (-2,2) {$223$};
\node at (0,2) {$232$};
\node at (2,2) {$322$};
\node at (0,0) {$222$};

\draw [shorten >=0.3cm,shorten <=0.3cm,->] (0,6) -- node [above,sloped,scale=0.7] {$\mathscr{P}_0\mathscr{P}_1$} (-2,4);
\draw [shorten >=0.3cm,shorten <=0.3cm,->] (0,6) -- node [above,sloped,scale=0.7] {$\mathscr{M}_0\mathscr{M}_1$} (0,4);
\draw [shorten >=0.3cm,shorten <=0.3cm,->] (0,6) -- node [above,sloped,scale=0.7] {$\tilde{\mathscr{P}}_0\tilde{\mathscr{P}}_1$} (2,4);
\draw [shorten >=0.3cm,shorten <=0.3cm,->] (-2,4) -- node [above,sloped,scale=0.7] {$P_0P_1$} (-2,2);
\draw [shorten >=0.3cm,shorten <=0.3cm,->] (0,6) to [bend right=20] node [above,sloped,scale=0.7,pos=0.75] {$\xi$} (0,2);
\draw [shorten >=0.3cm,shorten <=0.3cm,->] (0,4) to [bend left=20] node [above,sloped,scale=0.7,pos=0.25] {$\eta$} (0,0);
\draw [shorten >=0.3cm,shorten <=0.3cm,->] (0,4) -- node [above,sloped,scale=0.7,pos=0.35] {$\tilde{P}_0\tilde{P}_1$} (-2,2);
\draw [shorten >=0.3cm,shorten <=0.3cm,->] (0,4) -- node [above,sloped,scale=0.7,pos=0.35] {$\tilde{M}_0\tilde{M}_1$} (2,2);
\draw [shorten >=0.3cm,shorten <=0.3cm,->] (2,4) -- node [above,sloped,scale=0.7] {$\bar{M}_0\bar{M}_1$} (2,2);
\draw [shorten >=0.3cm,shorten <=0.3cm,->] (-2,4) -- node [above,sloped,scale=0.7] {$\eta$} (0,0);
\draw [shorten >=0.3cm,shorten <=0.3cm,->] (2,4) -- node [above,sloped,scale=0.7] {$\eta$} (0,0);
\draw [shorten >=0.3cm,shorten <=0.3cm,->] (0,2) -- node [below,sloped,scale=0.7] {$\tilde{p}_0\tilde{p}_1$} (0,0);
\end{tikzpicture}
\caption{The framed flow category $\cC_3$.}
\label{fig:smashofMooreII}
\end{figure}
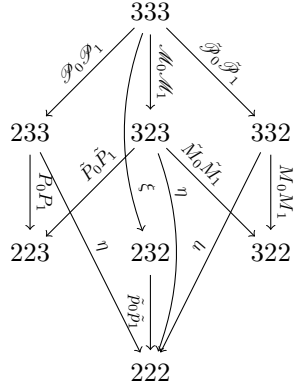

\begin{proposition}
	\label{prop:cC3_to_cC4}
	The framed flow category $\cC_3$ is move equivalent to the framed flow category $\cC_4$ depicted in Figure \ref{fig:smashofMooreIII}.
\end{proposition}

\begin{proof}
Starting from Figure \ref{fig:smashofMooreII} we can perform two $(-)$--handle slides, sliding $323$ over both $233$ and $332$. The moduli spaces $\cM(\overline{333},\overline{223})$ and $\cM(\overline{333},\overline{322})$ each consist of eight intervals, and $\cM(\overline{323},\overline{222})$ contains exactly three non-trivially framed circles, which are together Whitney trick equivalent to a single non-trivially framed circle. 
Note that in $\cM(\overline{333},\overline{322})$ the framing of the four new intervals is different from the framing of the new intervals in $\cM(\overline{333},\overline{223})$, but after performing the obvious Whitney tricks we get the flow category $\cC_4$ depicted in Figure \ref{fig:smashofMooreIII}.
\end{proof}

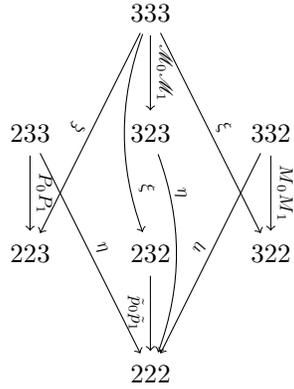
\begin{figure}[ht]
\begin{tikzpicture}[scale=0.8]
\node at (0,6) {$333$};
\node at (-2,4) {$233$};
\node at (0,4) {$323$};
\node at (2,4) {$332$};
\node at (-2,2) {$223$};
\node at (0,2) {$232$};
\node at (2,2) {$322$};
\node at (0,0) {$222$};

\draw [shorten >=0.3cm,shorten <=0.3cm,->] (0,6) -- node [above,sloped,scale=0.7] {$\mathscr{M}_0\mathscr{M}_1$} (0,4);
\draw [shorten >=0.3cm,shorten <=0.3cm,->] (-2,4) -- node [above,sloped,scale=0.7] {$P_0P_1$} (-2,2);
\draw [shorten >=0.3cm,shorten <=0.3cm,->] (0,6) to [bend right=20] node [above,sloped,scale=0.7,pos=0.75] {$\xi$} (0,2);
\draw [shorten >=0.3cm,shorten <=0.3cm,->] (0,4) to [bend left=20] node [above,sloped,scale=0.7,pos=0.25] {$\eta$} (0,0);
\draw [shorten >=0.3cm,shorten <=0.3cm,->] (0,6) -- node [above,sloped,scale=0.7] {$\xi$} (-2,2);
\draw [shorten >=0.3cm,shorten <=0.3cm,->] (0,6) -- node [above,sloped,scale=0.7] {$\xi$} (2,2);
\draw [shorten >=0.3cm,shorten <=0.3cm,->] (2,4) -- node [above,sloped,scale=0.7] {$\bar{M}_0\bar{M}_1$} (2,2);
\draw [shorten >=0.3cm,shorten <=0.3cm,->] (-2,4) -- node [above,sloped,scale=0.7] {$\eta$} (0,0);
\draw [shorten >=0.3cm,shorten <=0.3cm,->] (2,4) -- node [above,sloped,scale=0.7] {$\eta$} (0,0);
\draw [shorten >=0.3cm,shorten <=0.3cm,->] (0,2) -- node [below,sloped,scale=0.7] {$\tilde{p}_0\tilde{p}_1$} (0,0);
\end{tikzpicture}
\caption{The framed flow category $\cC_4$.  Note that the associated cochain complex is in Smith normal form (see Section \ref{sec:smith}).}
\label{fig:smashofMooreIII}
\end{figure}

\begin{proposition}
	\label{prop:cC4_to_cC5}
	The framed flow category $\cC_4$ is move equivalent to the framed flow category $\cC_5$ depicted in Figure \ref{fig:smashofMooreIV}.
\end{proposition}

\begin{proof}
Finally, we slide $332$ over $323$, $233$ over $323$, and then $232$ over $223$ and $322$. Again, using the extended Whitney trick on the 1-dimensional moduli spaces leads to the flow category $\cC_5$ depicted in Figure \ref{fig:smashofMooreIV}.
\end{proof}

\begin{figure}[ht]
\begin{tikzpicture}[scale=0.8]
\node at (0,6) {$333$};
\node at (-2,4) {$233$};
\node at (0,4) {$323$};
\node at (2,4) {$332$};
\node at (-2,2) {$223$};
\node at (0,2) {$232$};
\node at (2,2) {$322$};
\node at (0,0) {$222$};

\draw [shorten >=0.3cm,shorten <=0.3cm,->] (0,6) -- node [above,sloped,scale=0.7] {$\mathscr{M}_0\mathscr{M}_1$} (0,4);
\draw [shorten >=0.3cm,shorten <=0.3cm,->] (-2,4) -- node [above,sloped,scale=0.7] {$P_0P_1$} (-2,2);
\draw [shorten >=0.3cm,shorten <=0.3cm,->] (0,6) to [bend right=20] node [above,sloped,scale=0.7,pos=0.75] {$\xi$} (0,2);
\draw [shorten >=0.3cm,shorten <=0.3cm,->] (0,4) to [bend left=20] node [above,sloped,scale=0.7,pos=0.25] {$\eta$} (0,0);
\draw [shorten >=0.3cm,shorten <=0.3cm,->] (2,4) -- node [above,sloped,scale=0.7] {$\bar{M}_0\bar{M}_1$} (2,2);
\draw [shorten >=0.3cm,shorten <=0.3cm,->] (0,2) -- node [below,sloped,scale=0.7] {$\tilde{p}_0\tilde{p}_1$} (0,0);
\end{tikzpicture}
\caption{The flow category $\cC_5$.  The middle summand corresponds to a \emph{special cyclic word} in the sense of Baues-Hennes \cite{MR1113684}.}
\label{fig:smashofMooreIV}
\end{figure}
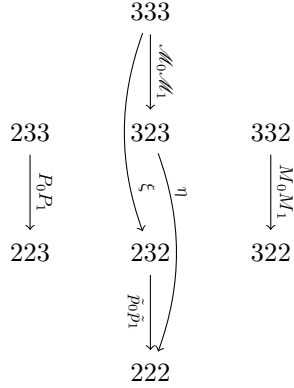

The flow category $\cC_5$ clearly has two wedge-summands $M(\Z/2\Z,7)$, and the remaining four objects form a stable space with the property that $\Sq^2\Sq^1=\Sq^1\Sq^2$ is non-trivial. It follows from the Decomposition Theorem of \cite[Theorem 3.9]{MR1113684} that the space corresponding to this category is the space called $X(\xi^2\eta_2,\id)$, where $\id\colon V \to V$ is the identity of the $1$--dimensional $\Z/2\Z$--vector space $V$.

We note that $(\xi^2\eta_2,\id)$ is a \emph{special cyclic word} in the sense of \cite{MR1113684}\footnote{In \cite{MR1113684}, the notation used is actually $\xi^1\eta_1$, while in \cite{MR1361886} it is $\xi^2\eta_2$.  We use the notation of \cite{MR1361886}.}, and the only indecomposable space in their list with $\Sq^2\Sq^1=\Sq^1\Sq^2$ non-trivial.

\section{A space level Smith normal form and further reductions}
\label{sec:smith}

It is discussed in \cite[Lemmas 3.25, 3.26]{MR3230817} that the stable homotopy type of $(\conf, \iota, \phi)$ is not affected by an isotopy of the framing $(\iota,\phi)$, or by increasing $\dd$ to $\dd'\geq \dd$. In addition to these basic modifications, we say the following are permissible \emph{flow category moves}:
\begin{enumerate}
\item Handle cancellation
\item Extended Whitney trick
\end{enumerate}(Note the handle slide of Subsection \ref{subsec:sliding} is based on handle cancellation, so does not need to be included as an extra flow category move.)

\begin{definition}$(\conf, \iota, \phi)$ and $(\conf', \iota',\phi')$ are \emph{directly move equivalent} if $(\conf', \iota',\phi')$ is the effect of a flow category move on $(\conf, \iota, \phi)$, or vice versa. We say $(\conf, \iota, \phi)$ and $(\conf', \iota',\phi')$ are \emph{move equivalent} if they are related by a finite sequence of directly move equivalent framed flow categories. Clearly, move equivalence is an equivalence relation on the set of framed flow categories.
\end{definition}

We now consider a general method to simplify a framed flow category within its move equivalence class.

Recall that a finite chain complex $C_*$ of finitely generated projective modules over a principal ideal domain $A$ can always be put into \emph{Smith normal form}. That is, there is a basis $C_r\cong U_r\oplus V_r\oplus W_r$ such that the differentials $d_r:C_{r+1}\to C_r$ vanish on $V_r\oplus W_r$ and the matrix of $d_r|_{U_r}$ has the form\[\left(\begin{matrix}0&D&0\end{matrix}\right)^t,\]where $D$ is injective and diagonal with $D_{ii}|D_{i+1\,i+1}$. Writing $m(r)=\text{rk}_AV_r$, $n(r)=\text{rk}_A W_r$, the Smith normal form basis presents the homology as \[H_r(C;A)\cong A/D_{11}A\oplus\dots\oplus A/D_{m(r)m(r)}A\oplus A^{n(r)}.\]Note that some diagonal entries may be equal to 1.

If we are moreover allowed to add or remove cancelling $A$--module generators in adjacent homological degrees (which will correspond to elementary expansions or contractions of the matrices of $d_r$), the Smith normal form can be changed so that the diagonal entries of $D$ are all prime powers. Elementary expansions and contractions result in chain homotopy equivalences, so the homology is not affected. In this changed form, the basis presents the unique primary decomposition of the homology modules. Call this modified type of basis the \emph{primary Smith normal form}.

Recall that any framed flow category $(\conf,\iota,\phi)$ determines a based chain complex $(C_*,d)$. The basis of $C_r$ is given by the objects $x$ of $\conf$ with $|x|=r$. If $|x|=r$ and $|y|=r-1$ then the $(x,y)$ entry in the matrix of $d_{r-1}$ is given by the signed count of the points in $\cM(x,y)$.

\begin{theorem}\label{thm:SNF}Any framed flow category $(\conf,\iota,\phi)$ is move equivalent to some framed flow category whose based chain complex $(C_*,d)$  is in primary Smith normal form and such that the number of points in any 0--dimensional moduli space is exactly the corresponding entry in the matrix of the differential $d$.
\end{theorem}

\begin{proof}Suppose $(\conf,\iota,\phi)$ is a framed flow category. Any 0--dimensional moduli space $\cM(x,y)$ consists of a certain number of positively framed points and a certain number of negatively framed points. By pairing a positive point with a negative point we may cancel them against each other using an extended Whitney trick. This will not affect any other 0--dimensional moduli spaces or the other components of $\cM(x,y)$. By repeatedly doing this, we may assume that every non-empty 0--dimensional moduli space consists entirely of positively framed points or entirely of negatively framed points. Write this as a positive or negative integer $n_{x,y}$, and if $\cM(x,y)$ is empty, set $n_{x,y}=0$. The chain complex of the framed flow category is then $C_*$, where $C_n$ is freely generated by the objects $x$ in grading level $n$ and $d_{n-1}x=\Sigma_{|y|={n-1}}n_{x,y}y$.

We wish to first put $C_*$ into Smith normal form. To do so we will show that certain basis changes for the $C_*$ can be realised via flow category moves. If $|x|=n$ then write $r_x$ for the row vector corresponding to $x$ in the matrix of $d_n$, and $c_x$ for the column vector in $d_{n-1}$. Now if $|x|=|y|$, then the reader may check that if we $(\pm)$--slide $x$ over $y$ this does not affect $r_x$ or $c_y$, but has the effect $r_y\mapsto r_y\mp r_x$ and $c_x\mapsto c_x\pm c_y$. Set $k=\min\{|x|\,|\,x\in\ob(\conf)\}$. We may now perform the row and column operations required to put $d_k$ into Smith normal form. As a result, $C_k\cong V_k\oplus W_k$ and $C_k\cong U_k\oplus V_k\oplus W_k$ as required. Next, observe that $\text{im} (d_{k+1})=V_{k+1}\subset \ker d_k$, so that performing the handle slides required to put $d_{k+1}$ into Smith normal form will not destroy the Smith normal form of $d_k$ (as there are no 0--dimensional moduli spaces between $V_{k+1}$ and $C_k$). We may now repeat this process on $d_r$, for increasing $r$, so that the whole of $C_*$ is in Smith normal form.

To modify the flow category so that the chain complex is in primary Smith normal form, we only need to know that for each $r$, we can make elementary matrix expansions or contractions of $d_r$, using flow category moves. But to make a matrix expansion, we simply introduce a pair of objects $x,y$ with $|x|=r+1$ and $|y|=r$ with $\cM(x,y)$ a single positively framed point. Clearly this has the required effect on the matrix of $d_r$, and $x$ can be cancelled against $y$ via handle cancellation, so introducing these points is a permissible flow category move. To make an elementary matrix contraction we perform a handle cancellation on the objects corresponding to the cancelling chain complex generators.
\end{proof}

Here is an easy corollary:

\begin{corollary}\label{cor:moore}Suppose $(\conf_1,\iota_1,\phi_1)$ and $(\conf_2,\iota_2,\phi_2)$ are framed flow categories such that there is a homotopy equivalence $\mathcal{X}(\conf_1)\simeq \mathcal{X}(\conf_2)$, and that for some $n\in\Z$ and all $k>0$ the reduced homology $\widetilde{H}_*(\conf_1;\Z/k\Z)$ is only supported in degrees $n$ and $n+1$ (in particular this means the flow categories have the stable homotopy type of a wedge of Moore spaces). Then $(\conf_1,\iota_1,\phi_1)$ and $(\conf_2,\iota_2,\phi_2)$ are move equivalent.
\end{corollary}

In fact we suggest that much more is true.

\begin{conjecture}\label{conj:main} If two framed flow categories determine the same stable homotopy type then they are move equivalent to one another.
\end{conjecture}

We provide more evidence for this conjecture in a forthcoming paper where we show that something similar to Corollary \ref{cor:moore} is true but with the reduced homology now possibly supported in degrees $n$, $n+1$, $n+2$ and $n+3$. To do this, we will show how to use move equivalence to reduce the stable homotopy types of such flow categories to those associated to the Chang \cite{MR0036508} and Baues-Hennes \cite{MR1113684} homotopy classifications.

\bibliographystyle{plain}
\def\MR#1{}

\bibliography{morsebiblio}

\end{document}